\documentclass[11pt,a4paper]{amsart}

\usepackage[dvipdfmx]{graphicx}
\usepackage{amsmath,amsthm,amssymb} 
\usepackage{here}

\newtheorem{defi}{Definition}
\newtheorem{thm}[defi]{Theorem}
\newtheorem*{thm*}{Theorem}
\newtheorem{prop}[defi]{Proposition}
\newtheorem*{prop*}{Proposition}
\newtheorem{lem}[defi]{Lemma}
\newtheorem{rem}[defi]{Remark}
\newtheorem*{rem*}{Remark}
\newtheorem{cor}[defi]{Corollary}
\newtheorem*{exa*}{Example}

\newtheorem*{fact*}{Fact}
\newtheorem*{claim*}{Claim}

\begin{document}

\title[$d$-minimal surfaces in $\mathbb{R}^{0,2,1}$]{$d$-minimal surfaces in three-dimensional singular semi-Euclidean space $\mathbb{R}^{0,2,1}$}
\author[Y. Sato]{Yuichiro Sato}
\address{Department of Mathematical Sciences,
Tokyo Metropolitan University,
Minami-Osawa 1-1, Hachioji, Tokyo, 192-0397, Japan.}
\email{satou-yuuichirou@ed.tmu.ac.jp}

\subjclass[2010]{Primary 53A10; Secondary 53A40}

\date{}

\keywords{minimal surface, isotropic geometry, semi-Riemannian geometry}

\begin{abstract}
In this paper, we study surfaces in singular semi-Euclidean space $\mathbb{R}^{0,2,1}$ endowed with a degenerate metric. We define $d$-minimal surfaces, and give a representation formula of Weierstrass type. Moreover, we prove that $d$-minimal surfaces in $\mathbb{R}^{0,2,1}$ and spacelike flat zero mean curvature (ZMC) surfaces in four-dimensional Minkowski space $\mathbb{R}^{4}_{1}$ are in one-to-one correspondence.
\end{abstract}

\maketitle 

\section{Introduction}
In this paper, we investigate surfaces in three-dimensional singular semi-Euclidean space with the signature $(0,2,1)$. 
The history of surface theory is very long, and the research has been studied. Minimal surfaces which attain stationary values for the volume functional of surfaces have many results of the research. In particular, they are characterized by having the mean curvature vector field which vanishes identically. Recently, Umehara and Yamada et al. (\cite{UY}, \cite{FSUY} and \cite{FKKRUY} etc.) study the zero mean curvature surfaces in three-dimensional Minkowski space actively. 
For such surfaces, they showed that singularities appear generically, and relate to the topology of surfaces. 

On the other hand, the author \cite{Sato} classified ruled minimal surfaces in semi-Euclidean space. As a consequence, we obtained that certain surfaces are included in three-dimensional subspaces whose metrics are degenerate forms. Inspired by this fact, we study the singular differential geometry, i.e. allow to have degenerate metrics. In particular, we consider the surface theory. We introduce a degenerate metric $dx^{2} + dy^{2}$ to three-dimensional vector space $\mathbb{R}^{3}$ with the coordinates $(x, y, z)$. We call the pair $(\mathbb{R}^{3}, dx^{2} + dy^{2})$ three-dimensional singular semi-Euclidean space with the signature $(0,2,1)$. It is denoted by $\mathbb{R}^{0,2,1}$. Let $M$ be a surface in $\mathbb{R}^{0,2,1}$. We assume that the induced metric of $M$ is non-degenerate. Actually, this degenerate geometry is equivalent to simply isotropic geometry which is one of the Cayley-Klein geometries. 
For isotropic geometry, the well-known reference is \cite{Sach}. 
We reformulate in terms of the geometry using metrics and connections. 

Here, we remark how to use the terms. 
First, in the canonical three-dimensional Euclidean space $\mathbb{R}^{3}$, surfaces whose mean curvature vanishes identically give stationary values for the volume functional. In a certain situation, its value is minimal, but not extreme in general. Historically, we call such surfaces \textit{minimal}.

Next, in three-dimensional Minkowski space $\mathbb{R}^{3}_{1}$, surfaces whose mean curvature vanishes identically change its name with respect to the cases of the induced metrics. When the induced metric is spacelike, i.e. Riemannian, we call such surfaces \textit{maximal}. This means that, when we consider the volume functional analogically, such surfaces always give maximal values unlike the Euclidean case. On the other hand, when timelike, i.e. Lorentzian, we simply call such surfaces \textit{minimal}. We should remark that timelike minimal surfaces give stationary values for the volume functional, but give neither minimal nor maximal values. When connected surfaces have the part of spacelike maximal surfaces and that of timelike minimal surfaces, we call such surfaces \textit{mixed type} (\cite{FKKRUY}).

In four-dimensional Minkowski space $\mathbb{R}^{4}_{1}$, surfaces whose mean curvature vector field vanishes identically are more complicated. Therefore, in order to treat uniformly, we call all such surfaces \textit{zero mean curvature} when the ambient space is $\mathbb{R}^{4}_{1}$. This is why we have to pay attention to the terminology.

In the section two, we recall the fundamental fact in semi-Riemannian geometry, and recall properties of non-degenerate submanifolds. In particular, we explain the singular semi-Euclidean space. 

In the section three, this is the main section. We define non-degenerate surfaces in $\mathbb{R}^{0,2,1}$ and study their properties in detail. In particular, $d$-minimal surfaces which we define are analogical objects to classical minimal surfaces. They are called \textit{isotropic minimal surfaces} in terms of simply isotropic geometry (\cite{Sach}). In addition to, we show a representation formula of Weierstrass type for $d$-minimal surfaces (Theorem \ref{thm4.2.2}), and claim that $d$-minimal surfaces allow to have isolated singularities. As an application, we prove that $d$-minimal surfaces and spacelike flat zero mean curvature (ZMC) surfaces in four-dimensional Minkowski space are in one-to-one correspondence (Corollary \ref{cor}). In particular, we see that there exist infinitely many spacelike flat ZMC surfaces in $\mathbb{R}^{4}_{1}$, which are not congruent each other. 

Because of these consequences, we see that spacelike flat ZMC surfaces in $\mathbb{R}^{4}_{1}$ are contained in a three-dimensional subspace endowed with a degenerate induced metric. However, we remark that it is the known fact by \cite{AP} and \cite{MWW}. 
And, local expressions are given by \cite{AP}, however, we study global expressions such as the representation formula and having singularities.

From Table \ref{tab:1}, we see that $d$-minimal surfaces in $\mathbb{R}^{0,2,1}$ have neutral properties between minimal surfaces in $\mathbb{R}^{3}$ and maximal surfaces in $\mathbb{R}^{3}_{1}$.
Regarding singularities, they do not appear on minimal surfaces. However, on maximal surfaces, cuspidal edges, swallowtails and cuspidal crosscaps appear in generic case. Refer to \cite{FSUY} in detail. On the other hand, for $d$-minimal surfaces, they allow to have isolated singularities. In this paper, these singularities are not classified. 

\section{Preliminaries}
In this section, we explain the fundamental properties for semi-Riemannian manifolds and their non-degenerate submanifolds.

\subsection{Semi-Riemannian manifolds}
Let $(M, g)$ be an $n$-dimensional semi-Riemannian manifold. For each $x \in M$ and a tangent vector $X \in T_{x}M$, we call $X$ 
\begin{eqnarray*}
\textit{spacelike} &:\Leftrightarrow& g(X, X) > 0 \ \textrm{or} \ X = 0, \\
\textit{timelike} &:\Leftrightarrow& g(X, X) < 0, \\
\textit{lightlike} \ (\textrm{or} \ \textit{null}) &:\Leftrightarrow& g(X, X) = 0.
\end{eqnarray*}
These are called \textit{causal properties} of tangent vectors (\cite{O}). 
As in the case of Riemannian manifolds, there exists uniquely a torsion-free, and metric connection $\nabla$ for a semi-Riemannian manifold. 
We call $\nabla$ the \textit{Levi-Civita connection} of $(M, g)$. Hereinafter, we consider that connections for semi-Riemannian manifolds are Levi-Civita connections. 

We define the \textit{curvature tensor field} $R$ of a semi-Riemannian manifold $(M, g)$ as 
\[ R(X, Y)Z := \nabla_{X} \nabla_{Y}Z - \nabla_{Y} \nabla_{X}Z - \nabla_{[X, Y]}Z \quad (X, Y, Z \in \Gamma(TM)). \]
Next, for each $x \in M$, let $P$ be a two-dimensional non-degenerate subspace of the tangent vector space $T_{x}M$, and let $\{X, Y\}$ be a basis of $P$. Then, we define the \textit{sectional curvature} $K(P)$ of $P$ as 
\[ K(P) := \frac{g(R(X, Y)Y, X)}{g(X, X)g(Y, Y) - g(X, Y)^{2}}, \]
where a subspace $P \subset T_{x}M$ is called \textit{non-degenerate} if the restriction on $P$ of $g$ is the non-degenerate form, and it called \textit{degenerate} if not so. In particular, when the dimension of $M$ is two, sectional curvatures are called \textit{Gaussian curvatures}. 
We denote the set consisting of smooth functions on $M$ by $C^{\infty}(M)$. For each $u \in C^{\infty}(M)$, we define the \textit{gradient vector field} $\textrm{grad} {u}$ of $u$ as 
\[ g(\textrm{grad} {u}, X) = du(X) \quad (\forall X \in \Gamma(TM)), \]
where $du$ denotes the exterior derivative of $u$. 
Next, for each $X \in \Gamma(TM)$, we define the \textit{divergence} $\textrm{div}{X}$ of $X$ as 
\[ \textrm{div}{X} := \textrm{tr}( (X_{1}, X_{2}) \mapsto g(\nabla_{X_{1}}X, X_{2}) ) \quad (X_{1}, X_{2} \in \Gamma(TM)). \]
For each $u \in C^{\infty}(M)$, we define the \textit{Laplacian} $\Delta_{g}{u}$ of $u$ with respect to $g$ as 
\[ \Delta_{g}{u} := \textrm{div} (\textrm{grad} {u}). \]
When $\Delta_{g}{u} \equiv 0$, we call a function $u$ \textit{harmonic}. 

When let $\{e_{1}, \cdots, e_{n}\}$ be a local orthonormal frame of $(M, g)$, the gradient vector field and the divergence respectively have the following local expressions 
\begin{eqnarray*}
\textrm{grad} {u} &=& \sum_{i=1}^{n} \epsilon_{i} du(e_{i})e_{i}, \\
\textrm{div}{X} &=& \sum_{i=1}^{n} \epsilon_{i} g(\nabla_{e_{i}}X, e_{i}),
\end{eqnarray*}
where $\epsilon_{i} = g(e_{i}, e_{i}) = \pm 1$. 

\subsection{Non-degenerate submanifolds}
Let $M$ be an $m$-dimensional manifold, and let $(N, \bar{g})$ be an $n$-dimensional semi-Riemannian manifold. We assume that a $C^{\infty}$-mapping $f : M \rightarrow N$ is an immersion.  
Then, we call $M$ an \textit{immersed submanifold} in $N$. In particular, when $f$ is injective, and $M$ is homeomorphic to the image $f(M)$ as the subspace of $N$, $M$ is said to be a \textit{embedded submanifold} in $N$. 

We denote the induced metric $f^{\ast}\bar{g}$ on $M$ by $g$. For semi-Riemannian manifolds, we remark that $g$ is not always non-degenerate even if $f$ is an immersion. When the induced metric $g$ is non-degenerate, we call $(M, g)$ a \textit{non-degenerate submanifold}, or a \textit{semi-Riemannian submanifold} of $(N, \bar{g})$. 

Hereinafter, when we describe submanifolds, unless otherwise noted, we consider immersed, non-degenerate submanifolds. Then, for each $x \in M$, a \textit{normal vector space} $T_{x}^{\bot}M$ is defined as 
\[ T_{x}^{\bot}M := \{ v \in T_{f(x)}N \ | \ \bar{g}(df_{x}(w), v) = 0 , \ \forall w \in T_{x}M \}. \]
We obtain a vector bundle $T^{\bot}M = \bigcup_{x \in M}T_{x}^{\bot}M$ of rank $(n - m)$ over $M$. This is called a \textit{normal bundle} of $M$. By this, for each $x \in M$, we have the orthogonal direct sum decomposition 
\[ T_{f(x)}N = T_{x}M \perp T_{x}^{\bot}M, \]
where $\perp$ stands for the orthogonal direct sum. 
In particular, we see that, as the orthogonal direct sum of vector bundles, it holds 
\[ f^{\ast}TN = TM \perp T^{\bot}M, \]
where $f^{\ast}TN$ is the pull-back bundle over $M$ by $f$. 
We denote the Levi-Civita connection of $(N, \bar{g})$ and that of $(M, g)$ by $\bar{\nabla}$ and $\nabla$ respectively. And, we define the set $\Gamma(T^{\bot}M)$ as the whole of smooth sections of the normal bundle $T^{\bot}M$. This section is said to be a \textit{normal vector field} particularly. 

For each $X, Y \in \Gamma(TM), \ \xi \in \Gamma(T^{\bot}M)$, by using the orthogonal direct sum decomposition given above, we have 
\begin{eqnarray}
\bar{\nabla}_{X}Y &=& \nabla_{X}Y + h(X, Y), \label{gauss_form} \\
\bar{\nabla}_{X}\xi &=& -A_{\xi}X + \nabla^{\bot}_{X}\xi, \label{weingarten_form}
\end{eqnarray}
where $h, A_{\xi}$ and $\nabla^{\bot}$ are called the \textit{second fundamental form}, the \textit{shape operator} with respect to $\xi$ and the \textit{normal connection} on $M$ respectively. We call the formula (\ref{gauss_form}) and (\ref{weingarten_form}) \textit{Gauss formula} and \textit{Weingarten formula} of $M$ respectively. 

\subsection{Singular semi-Euclidean spaces}
We define the $n$-dimensional \textit{singular semi-Euclidean space} with the signature $(p, q, r)$ as 
\[ \mathbb{R}^{p,q,r} := \left (\mathbb{R}^{n}, (\cdot , \cdot) = - \sum_{i=1}^{p} dx_{i}^{2} + \sum_{j=p+1}^{p+q} dx_{j}^{2} + \sum_{k=p+q+1}^{n} 0 dx_{k}^{2} \right ), \]
where $n=p+q+r$ and $(x_{1}, \cdots, x_{n})$ expresses the canonical coordinates on $\mathbb{R}^{n}$ (\cite{St}). We remark the following statement:
\begin{itemize}
\item When $r=0$, $\mathbb{R}^{p,q,0}$ is called \textit{semi-Euclidean space} having index $p$, and we denote it by $\mathbb{R}^{n}_{p}$.
\item When $p=r=0$, $\mathbb{R}^{0, n, 0} = \mathbb{R}^{n}_{0}$ is nothing but Euclidean space $\mathbb{R}^{n}$.
\end{itemize}

We remark that $r \geq 1$ if and only if the metric $(\cdot , \cdot )$ is degenerate. 
In isotropic geometry, the notation $\mathbb{R}^{0, n-1, 1}$ is also known as the simply isotropic $n$-space $\mathbb{I}^{n}$ (\cite{Sach}).
From now on, we state fundamental objects for a semi-Euclidean space and its non-degenerate submanifolds. 

For $n$-dimensional semi-Euclidean space $\mathbb{R}^{n}_{p}$ having index $p \ (0 \leq p \leq q)$, we assume that the semi-Euclidean metric is given by 
\[ \langle \cdot, \cdot \rangle_{p} := -\sum_{i=1}^{p} dx_{i}^2 + \sum_{j=p+1}^{n} dx_{j}^2, \]
where $(x_{1}, \cdots, x_{n})$ is the canonical coordinates of $\mathbb{R}^{n}_{p}$. 
A non-zero vector $v$ in $\mathbb{R}^{n}_{p}$ is called \textit{spacelike}, \textit{timelike} and \textit{lightlike} if it satisfies $\langle v, v \rangle_{p} > 0, \langle v, v \rangle_{p} < 0$ and $\langle v, v \rangle_{p} = 0$ respectively. 

The $n$-dimensional semi-Euclidean space $\mathbb{R}^{n}_{1}$ having index one is said to be the $n$-dimensional \textit{Minkowski space}. Moreover, the four-dimensional Minkowski space is closely related to the physics as the flat spacetime model. 

Let $M$ be an $m$-dimensional non-degenerate submanifold in $\mathbb{R}^{n}_{p}$. We denote the Levi-Civita connections for $\mathbb{R}^{n}_{p}$ and $M$ by $\bar{\nabla}$ and $\nabla$ respectively. 
And, let $X, Y, Z, W$ and $\xi, \eta$ be tangent vector fields and normal vector fields on $M$ respectively. 

\textit{Gauss equation}, \textit{Codazzi equation} and \textit{Ricci equation} of $M$ are given by the following 
\begin{eqnarray}
\langle R(X, Y)Z, W \rangle_{p} &=& \langle h(Y, Z), h(X, W) \rangle_{p} - \langle h(X, Z), h(Y, W) \rangle_{p}, \label{gauss_eq} \\
(\nabla_{X}h)(Y, Z) &=& (\nabla_{Y}h)(X, Z), \label{codazzi_eq} \\
\langle R^{\bot}(X, Y)\xi , \eta \rangle_{p} &=& \langle [A_{\xi}, A_{\eta}]X, Y \rangle_{p},  \label{ricci_eq}
\end{eqnarray}
where $R$ and $R^{\bot}$ are curvature tensor fields with respect to connections $\nabla$ and $\nabla^{\bot}$ respectively, and $\nabla_{X}h$ is the covariant derivative of the second fundamental form $h$ for the tangent vector field $X$, i.e. it is defined by
\[ (\nabla_{X}h)(Y, Z) = \bar{\nabla}_{X}h(Y, Z) - h(\nabla_{X}Y, Z) - h(Y, \nabla_{X}Z). \]
Moreover, the normal bundle $T^{\bot}M$ of $M$ is called \textit{flat} if $R^{\bot} \equiv 0$. 

Let $\{ e_{1}, \cdots, e_{m} \}$ be a local orthonormal frame of the tangent bundle $TM$, and let $\{ e_{m+1}, \cdots, e_{n} \}$ be a local orthonormal frame of the normal bundle $T^{\bot}M$. In addition to, setting $\epsilon_{A}:=\langle e_{A}, e_{A} \rangle_{p} = \pm 1$, we use the following range of indices: 
\[ 1 \leq A, B, C, \dots \leq n, \quad 1 \leq i, j, k, \dots \leq m, \quad m+1 \leq \alpha, \beta, \gamma, \dots \leq n.  \]
We denote the connection form of $\nabla$ associated $\{ e_{1}, \cdots, e_{m} \}$ by $\{ \omega_{i}^{j} \}$, and we denote the connection form of $\nabla^{\bot}$ associated $\{ e_{m+1}, \cdots, e_{n} \}$ by $\{ \omega_{\beta}^{\alpha} \}$. Then, from Gauss formula (\ref{gauss_form}) and Weingarten formula (\ref{weingarten_form}), we have 
\begin{eqnarray}
\bar{\nabla}_{e_{k}}e_{i} &=& \sum_{j=1}^{m} \epsilon_{j}\omega_{i}^{j}(e_{k})e_{j} + \sum_{\alpha=m+1}^{n}\epsilon_{\alpha}h_{ik}^{\alpha}e_{\alpha}, \\
\bar{\nabla}_{e_{k}}e_{\beta} &=& - \sum_{j=1}^{m}\epsilon_{j}h_{kj}^{\beta}e_{j} + \sum_{\alpha=m+1}^{n}\epsilon_{\alpha}\omega_{\beta}^{\alpha}(e_{k})e_{\alpha},
\end{eqnarray}
where $h_{ij}^{\alpha}$ are coefficients of the second fundamental form. Moreover, we see that the mean curvature vector field $\vec{H}$ of $M$ is expressed by 
\begin{equation}
\vec{H} = \frac{1}{m}\sum_{\alpha=m+1}^{n} \epsilon_{\alpha} \textrm{tr}A_{\alpha}e_{\alpha},  \label{meancurv}
\end{equation}
where $\textrm{tr}A_{\alpha}$ is the trace of the shape operator $A_{e_{\alpha}}$ with respect to $e_{\alpha}$, i.e. $\textrm{tr}A_{\alpha} = \sum_{i=1}^{m} \epsilon_{i}h_{ii}^{\alpha}$.

\section{$d$-minimal surfaces in singular semi-Euclidean space}
In this section, we consider three-dimensional singular semi-Euclidean space with the signature $(0, 2, 1)$. We define $\mathbb{R}^{0, 2, 1}$ as
\[ \mathbb{R}^{0, 2, 1} := \left ( \mathbb{R}^{3}, (\cdot, \cdot) = dx^{2} + dy^{2} \right ), \]
where let $(x,y,z)$ be the canonical coordinates. And, we study surfaces in $\mathbb{R}^{0, 2, 1}$.

\subsection{Preparations}
Let $M$ be a two-dimensional manifold, let $f : M \rightarrow \mathbb{R}^{0,2,1}$ be a $C^{\infty}$-immersion and let $g$ be the induced metric by $f$. We assume that the metric $g$ is a positive definite symmetric bilinear form. And, we call $f$ a \textit{non-degenerate immersion} or a \textit{non-degenerate surface}. Then, for each $x \in M$, a \textit{normal vector space} $T^{\bot}_{x}M$ is defined by
\[ T_{x}^{\bot}M := \{ \xi \in \mathbb{R}^{3} \ | \ (df_{x}(v), \xi) = 0 , \ \forall v \in T_{x}M \} = \textrm{span}
_{\mathbb{R}}\{(0,0,1)\}, \]
and we have a vector bundle of rank one over $M$
\[ T^{\bot}M = \bigcup_{x \in M}T_{x}^{\bot}M. \]
Therefore, we obtain an orthogonal direct sum decomposition
\[ T_{f(x)}\mathbb{R}^{3} = T_{x}M \perp T_{x}^{\bot}M \]
for each $x \in M$. In particular, we see, as a vector bundle decomposition,
\[ f^{\ast}T\mathbb{R}^{3} = TM \perp T^{\bot}M, \]
where $TM$ is the tangent bundle over $M$ and $f^{\ast}T\mathbb{R}^{3}$ is the pull-back bundle by $f$ over $M$.

\begin{prop} \rm
We get an isomorphism as vector bundle
\[ T^{\bot}M \cong M \times \mathbb{R}. \]
\end{prop}

\begin{proof}
We can take $\xi = (0,0,1) \in \Gamma(T^{\bot}M)$ as a non-vanishing global section. 
\end{proof}

\begin{rem} \rm
For three-dimensional singular semi-Euclidean space with the signature $(p, q, r)$, where $p+q+r=3, \ r \geq 1, \ p \leq q$, we can define non-degenerate surfaces when $r=1$, i.e.
\[ (p, q, r) = (0,2,1), \ (1,1,1). \]
When $r \geq 2$, the metric induced on surfaces is degenerate.
We remark that $\mathbb{R}^{1,1,1}$ is equivalent to the pseudo-isotropic $3$-space $\mathbb{I}^{3}_{1}$ (Refer to \cite{Silva1}, \cite{Silva2} and \cite{Ay1}).
And, as a notation, we define 
\[ |v| := \sqrt{(v,v)} = \sqrt{v_{1}^{2} + v_{2}^{2}} \]
for a vector $v = (v_{1}, v_{2}, v_{3}) \in \mathbb{R}^{0,2,1}$.
\end{rem}

Next, we recall affine differential geometry (\cite{NS}). 
Let $(\mathbb{R}^{n+1}, d)$ be $(n+1)$-dimensional Euclidean space with the canonical connection $d$ and $M$ be an $n$-dimensional manifold. A $C^{\infty}$-immersion $f : M \rightarrow \mathbb{R}^{n+1}$ is an \textit{affine immersion} if for any $x \in M$ there exists a neighborhood $U$ at $x$ and a vector field $\xi$ on $U$ over $\mathbb{R}^{n+1}$ such that 
\[ T_{f(y)}\mathbb{R}^{n+1} = T_{y}M \oplus \mathbb{R}\xi_{y} \quad (\forall y \in U), \]
where $\oplus$ stands for the direct sum. 
In particular, when there exists $\xi$ globally on $M$, it is called a \textit{transversally vector field} on $M$. Then, a torsion-free connection $\nabla$ is induced on $M$, and it satisfies 
\[ d_{X}Y = \nabla_{X}Y + h(X,Y)\xi \]
for any $X, Y \in \Gamma(TM)$.
This implies that $h$ is a $(0, 2)$-type symmetric tensor field over $M$, and we call $h$ an \textit{affine fundamental form} (with respect to $\xi$). In affine differential geometry, we often assume that $h$ is non-degenerate. Moreover, let $f : M \rightarrow \mathbb{R}^{n+1}$ be an affine immersion and let $\xi$ be its transversally vector field. we call $\xi$ \textit{equiaffine} when 
\[ \forall X \in \Gamma(TM), \quad d_{X}\xi \in \Gamma(TM). \]
Then, $f$ is called an \textit{equiaffine immersion}.

In terms of affine differential geometry, we see the following proposition.

\begin{prop} \rm
Let $M$ be a two-dimensional manifold. 
A non-degenerate immersion $f : M \rightarrow \mathbb{R}^{0,2,1}$ is an equiaffine immersion whose transversally vector field over $M$ is $\xi \equiv (0, 0, 1)$.
\end{prop}

\begin{proof}
By using the orthogonal direct sum $f^{\ast}T\mathbb{R}^{3} = TM \perp T^{\bot}M$, and $d_{X}\xi = 0$ for all $X \in \Gamma(TM)$, the proof is completed.
\end{proof}

Hereinafter, let $\xi$ be the constant vector field $\xi = (0, 0, 1)$ and let $d$ be the canonical connection as a linear connection, i.e. for all $X, Y \in \Gamma(T\mathbb{R}^{0,2,1})$, identifying $Y$ with the vector-valued function $Y=(Y_{1}, Y_{2}, Y_{3})$,
\[ d_{X}Y := dX(Y) = (X(Y_{1}), X(Y_{2}), X(Y_{3})). \]
Then, the connection $d$ is torsion-free and preserves the degenerate metric $( \cdotp, \cdotp )$. Thus, the connection $d$ plays the role of the Levi-Civita connection.

We define the automorphism group $\textrm{Aut}(\mathbb{R}^{0, 2, 1}, d)$ with respect to $\mathbb{R}^{0, 2, 1}$ and $d$ as 
\begin{eqnarray*}
\textrm{Aut}(\mathbb{R}^{0, 2, 1}, d) &:=& 
\{ A \in \textrm{Diff}(\mathbb{R}^{3}) \ | \ A^{\ast}d = d, \ A^{\ast}(\cdot, \cdot) = (\cdot, \cdot) \} \\
&=& O(0, 2, 1) \ltimes \mathbb{R}^3,
\end{eqnarray*}
where $\textrm{Diff}(\mathbb{R}^{3})$ is the diffeomorphism group of $\mathbb{R}^{3}$ and
\[ O(0, 2, 1):= \left. \left \{
    \left(
    \begin{array}{ccc}
      \multicolumn{2}{c}{\raisebox{-2.0ex}{\LARGE{\textit{T}}}} & 0 \\
       & &0 \\
      a & b& c
    \end{array}
  \right)
 \ \right | \ a, b, c \in \mathbb{R}, \ c \neq 0, \ T \in O(2) \right \}. \]
We call $\textrm{Aut}(\mathbb{R}^{0, 2, 1}, d)$ an \textit{affine isometry group}. In particular, $\textrm{Aut}(\mathbb{R}^{0, 2, 1}, d) $ is a seven-dimensional Lie group. From the view of Cayley-Klein geometry, this automorphism group is nothing but the simply isotropic rigid motion group (\cite{Silva1}). 

By using the decomposition $f^{\ast}T\mathbb{R}^{3} = TM \bot T^{\bot}M$, for each $X, Y \in \Gamma(TM)$, $\alpha \xi \in \Gamma(T^{\bot}M) \ (\alpha \in C^{\infty}(M))$, we have
\begin{eqnarray*}
d_{X}Y &=& \nabla_{X}Y + h(X,Y)\xi, \\
d_{X}(\alpha \xi) &=& X(\alpha)\xi.
\end{eqnarray*}
Then, we see that the connection $\nabla$ is the Levi-Civita connection with respect to the induced metric $g$ on $M$. And, we call the given affine fundamental form $h$ a \textit{second fundamental form} of the non-degenerate immersion $f$.

For all $X, Y, Z \in \Gamma(TM)$, since the connection $d$ is flat, we obtain
\[ 0 = {}^{d}R(X,Y)Z = {}^{\nabla}R(X,Y)Z + \{ (\nabla_{X}h)(Y,Z) - (\nabla_{Y}h)(X,Z) \}\xi, \]
where ${}^{d}R$ and ${}^{\nabla}R$ are the curvature tensor fields for $d$ and $\nabla$ respectively, and we define $(\nabla_{X}h)(Y,Z):=X(h(Y,Z)) - h(\nabla_{X}Y,Z) - h(Y, \nabla_{X}Z)$. Therefore, we get
\begin{eqnarray}
{}^{\nabla}R &\equiv& 0, \label{flat} \\
(\nabla_{X}h)(Y,Z) &=& (\nabla_{Y}h)(X,Z). \label{gauss-codazzi}
\end{eqnarray}
The formula (\ref{flat}) implies that the non-degenerate surface is always flat, and we call the formula (\ref{gauss-codazzi}) \textit{Gauss-Codazzi equation} of the non-degenerate surface. These formulas (\ref{flat}) and (\ref{gauss-codazzi}) were obtained by Sachs in \cite{Sach}. 

Let $f : M \rightarrow \mathbb{R}^{0, 2, 1}$ be a non-degenerate immersion. Then, the image of $f$ is locally expressed by the form of a graph surface $\{(u, v, F(u, v)) \in \mathbb{R}^{0, 2, 1} \ | \ (u, v) \in U \}$, where $F$ is a smooth function on an open subset $U \subset \mathbb{R}^2$.

Next, we define some classes for non-degenerate surfaces.
\begin{itemize}
\item[(i)] \textit{$d$-totally geodesic surface} :$\Leftrightarrow$ {the second fundamental form}$ \ h \equiv 0$,
\item[(ii)] \textit{$d$-totally umbilical surface} :$\Leftrightarrow$ $\exists \lambda \in C^{\infty}(M)$ s.t. $h = \lambda g$,
\item[(iii)] \textit{$d$-minimal surface} :$\Leftrightarrow$ $\mathcal{H}:=\dfrac{1}{2} \textrm{tr}_{g}h= \dfrac{1}{2} g^{ij}h_{ij} = 0$,
\end{itemize}
where $g^{ij}$ is the components of the inverse matrix of $(g_{ij})_{1 \leq i,j \leq 2}$ and $h_{ij}$ is the coefficients of the second fundamental form $h$. We call $\mathcal{H}$ the \textit{mean curvature} of the non-degenerate surface.
For (ii), we remark that (ii) is equivalent to (i) when $\lambda=0$.

\begin{prop} \rm
Let $M$ be a two-dimensional manifold, and 
let $f : M \rightarrow \mathbb{R}^{0,2,1}$ be connected, not $d$-totally geodesic and $d$-totally umbilical surface, that is, there exists a function $\lambda \in C^{\infty}(M)$ such that $h=\lambda g$ and $\lambda \neq 0$. Then, $\lambda$ is a constant function, and the image of $f$ is an open subset of a paraboloid of revolution
\[ \left. \left \{\left (u,v,\frac{\lambda}{2}(u^{2}+v^{2})+Au+Bv+C \right ) \in \mathbb{R}^{3} \ \right | \ (u,v) \in \mathbb{R}^{2} \right \}, \]
where $A, B, C \in \mathbb{R}$ are constant. In particular, it is, up to affine isometry, an open subset of 
\[ \left. \left \{(u, v, u^{2}+v^{2}) \in \mathbb{R}^{3} \ \right | \ (u,v) \in \mathbb{R}^{2} \right \}. \]
\end{prop}

\begin{proof}
Since non-degenerate surfaces satisfy Gauss-codazzi equation (\ref{gauss-codazzi}), the function $\lambda$ is a constant. 
Let $g$ be the induced metric by $f$ and let $h$ be its second fundamental form. From the assumption, there exists a non-zero constant number $\lambda \in \mathbb{R}$ such that $h = \lambda g$. Since $f$ is the non-degenerate immersion, for each point of $M$, there exists a coordinate neighborhood $\{U ; (u, v) \}$ such that 
\[  f(u, v) = (u, v, \varphi(u, v) ) \in \mathbb{R}^{0,2,1}, \]
where $\varphi$ is a $C^{\infty}$-function on $U$. Then, we get
\[ h_{11} = \varphi_{uu}, \ h_{12} = \varphi_{uv}, \ h_{22} = \varphi_{vv}. \]
Therefore, since we have
\[ \varphi_{uu} = \lambda g_{11} = \lambda, \ \varphi_{uv} = \lambda g_{12} = 0, \ \varphi_{vv} = \lambda g_{22} = \lambda, \]
there exist constant numbers $A, B, C \in \mathbb{R}$ such that 
\[ \varphi(u, v) = \frac{\lambda}{2}(u^{2} + v^{2}) + Au + Bv + C. \]
Finally, gluing these pieces of surface in the whole of $M$, we obtain the consequence.
\end{proof}

Here, we define a \textit{relative Gaussian curvature} $\mathcal{K}$ which is introduced in \cite{Sach} as
\[ \mathcal{K} := \frac{\det{h}}{\det{g}} \in C^{\infty}(M). \]
This quantity expresses the shape of the non-degenerate surface when we look from the ambient space $\mathbb{R}^{3}$. However, the canonical Gaussian curvature, i.e. the sectional curvature of two-dimensional Riemannian manifolds with respect to the induced metric, identically vanishes.

\begin{prop} [\cite{Sach}, Definition 8.11] \rm
Let $M$ be a two-dimensional manifold, and 
let $f : M \rightarrow \mathbb{R}^{0,2,1}$ be a non-degenerate immersion. Let $\mathcal{K}$ be its relative Gaussian curvature. Then, in a sense of surface theory in the canonical Euclidean space $\mathbb{R}^{3}$, we have
\begin{eqnarray*}
\mathcal{K}(x) > 0 &\Longleftrightarrow& x : \textrm{elliptic point}, \\
\mathcal{K}(x) < 0 &\Longleftrightarrow& x : \textrm{hyperbolic point}, \\
\mathcal{K}(x) = 0 &\Longleftrightarrow& x : \textrm{parabolic point},
\end{eqnarray*}
for each $x \in M$. However, when we consider $f$ as an immersion to Euclidean space $\mathbb{R}^{3}$, the canonical Gaussian curvature do not correspond with the relative Gaussian curvature in general.
\end{prop}

\begin{rem} \rm
We consider the sign of the relative Gaussian curvature for some surfaces.
First, for $d$-totally geodesic surfaces, since we have $h = 0$ by definition, it holds
\[ \mathcal{K} =  \frac{\det{h}}{\det{g}} \equiv 0. \]
Next, for $d$-totally umbilical surfaces, we have, by definition, there exists a constant number $\lambda \in \mathbb{R}$ such that $h = \lambda g$. We assume $\lambda \neq 0$. Then, we obtain 
\[ \mathcal{K} = \frac{\det{h}}{\det{g}} = \frac{\lambda^{2} \det{g}}{\det{g}} = \lambda^{2} > 0, \]
that is, all points are elliptic.
Finally, for $d$-minimal surfaces, we make use of isothermal coordinates, that is, we choose the coordinates in which the coefficients of the induced metric hold
\[ g_{11} = g_{22} > 0, \quad g_{12} = 0. \]
Then, since the mean curvature identically vanishes, we have
\[ 2\mathcal{H} = \textrm{tr}_{g}h = \frac{g_{22}h_{11} + g_{11}h_{22}}{g_{11}g_{22}} = \frac{h_{11} + h_{22}}{g_{11}} \equiv 0. \]
Moreover, by using $h_{22} = -h_{11}$, we obtain
\[ \mathcal{K} = \frac{\det{h}}{\det{g}} = \frac{h_{11}h_{22} - h_{12}^{2}}{g_{11}g_{22}} = -\frac{h_{11}^{2} + h_{12}^{2}}{g_{11}^{2}} \leq 0, \]
that is, almost all points are hyperbolic.
\end{rem}

Here, we give some descriptions for curves in $\mathbb{R}^{0, 2, 1}$. 
For a connected open interval $I \subset \mathbb{R}$, let $c$ be a $C^{\infty}$-map $c : I \rightarrow \mathbb{R}^{0, 2, 1}$. We call $c$ a \textit{curve} in $\mathbb{R}^{0, 2, 1}$. Moreover, we call $c$ a \textit{regular} curve if it holds
\[ \forall t \in I, \ c'(t) \neq 0. \]
Next, let $\pi$ be the projection to $xy$-plane, i.e.
\[ \pi : \mathbb{R}^{0, 2, 1} \ni (x, y, z) \mapsto (x, y) \in \mathbb{R}^{2}. \]
And, we call a parameter $s$ of a curve $c = c(s)$ \textit{arc-length} if it holds
\[ |c'(s)| \equiv 1. \]
Then, we obtain the following propositions.

\begin{prop} \label{arc_length} \rm
Let $c = c(t) \ (t \in I)$ be a regular curve in $\mathbb{R}^{0, 2, 1}$. The following are equivalent:
\begin{itemize}
\item[(i)] The curve $c = c(t)$ admits an arc-length parameter.
\item[(ii)] For all $t \in I$, it holds $|c'(t)| > 0$.
\item[(iii)] The mapping $\pi \circ c$ is regular as a planar curve in $\mathbb{R}^{2}$.
\end{itemize}
\end{prop}

\begin{proof}
Easy calculations. 
\end{proof}

We call a regular curve $c = c(t) \ (t \in I)$ in  $\mathbb{R}^{0, 2, 1}$ \textit{null} if it holds
\[ |c'(t)| \equiv 0. \]

\begin{prop} \label{null_curve} \rm
A regular curve $c : I \rightarrow \mathbb{R}^{0,2,1}$ is null if and only if it is a spacial line which is parallel with the $z$-axis.
\end{prop}

\begin{proof}
Easy calculations. 
\end{proof}

\begin{prop} \label{kindofsurf} \rm 
For any connected surfaces in $\mathbb{R}^{0, 2, 1}$,
\begin{itemize}
\item[(0)] $d$-totally geodesic surfaces in $\mathbb{R}^{0,2,1}$ are non-degenerate planes only (\cite{Sach}, Theorem 9.4).
\item[(1)] a graph surface in $\mathbb{R}^{0,2,1}$
\[ \{ (u, v, f(u,v)) \in \mathbb{R}^{0,2,1} \ | \ (u,v) \in U \subset \mathbb{R}^2 \} \]
is $d$-minimal if and only if $f$ is a harmonic function on $U$ (\cite{Sach}, Eq. (9.31)).
\item[(2)] non-planar, ruled $d$-minimal surfaces in $\mathbb{R}^{0,2,1}$ are locally, up to affine isometry, open subset of 
 \begin{itemize}
 \item[(a)] $f(u,v) = (v\cos{u}, v\sin{u}, u)$ (refer to Figure \ref{fig:9}),
 \item[(b)] $f(u,v) = (u, v, uv)$ (refer to Figure \ref{fig:10}),
 \end{itemize}
where $(u, v) \in \mathbb{R}^{2}$ (\cite{Sato}, Theorem 6).
\item[(3)] non-planar, $d$-minimal rotational surfaces in $\mathbb{R}^{0,2,1}$ are locally, up to affine isometry, open subset of 
\[ f(u, v) = (e^{u}\cos{v}, e^{u}\sin{v}, u) \]
(refer to Figure \ref{fig:11}), where the rotational surfaces mean the rotation group, which acts on the $xy$-plane, $SO(2)$-invariant surfaces.
\end{itemize}
\end{prop}

\begin{proof}
(0) and (1) are proved by easy calculations. 

In case of $(2)$, we apply the method of classification described by \cite{Sato}.
Since we have the fact that the induced metrics of non-degenerate immersions are positive definite, ruled $d$-minimal surfaces of cylinder type are planes only. Thus, we have only to investigate the case of non-cylinder type. Let curves $\gamma(s)$ and $x(s)$ be a direction curve and a base curve of the given ruled surface respectively. Since we consider the case of non-cylinder, the direction curve $\gamma$ is regular. When $|\gamma'| \neq 0$, we can take the arc-length parameter of $\gamma$ from Proposition \ref{arc_length}. Then, we may set $\eta := |\gamma'|^{2} = (\gamma', \gamma') \equiv 1$ or $0$ in generic.

When $\eta \equiv 1$, we see that the direction curve is $\gamma(s) = \cos{s}e_{1} + \sin{s}e_{2}$, where vectors $e_{1}, e_{2} \in \mathbb{R}^{0, 2, 1}$ satisfy $|e_{1}| = |e_{2}| = 1, \ (e_{1}, e_{2}) = 0$. Therefore, it holds that $e_{1}, e_{2}$ are tangent vectors. If we assume that $(\gamma'(s), x'(s)) \equiv 0$, it holds $(x'(s), x'(s)) \equiv 0$. Thus, regarding the Table $2$ in \cite{Sato}, the case of (i) does not exist. Since the case of (iii) is reduced to the case of (ii), we have only to consider the case of (ii). For the case of (ii), by an affine isometry, we have
\[ f(s, t) = (t\cos{s}, t\sin{s}, s). \]

Next, when $\eta \equiv 0$, we see that $\gamma'$ is a normal vector. Thus, it holds $(\gamma'(s), x'(s)) \equiv 0$. Therefore, we have only to consider the case of (v). Then, by an affine isometry, we have
\[ f(s, t) = (s, t, st). \]

In case of $(3)$, we explain the meaning of $SO(2)$-invariant firstly. It is well-known that
\[ SO(2) = 
\left \{ \left (
 \begin{array}{cc}
   \cos{\theta} & -\sin{\theta} \\
   \sin{\theta} & \cos{\theta}
 \end{array}
\right ) \in M_{2}(\mathbb{R}) \ 
\middle | \ \theta \in \mathbb{R}
\right \}.
\]
We realize $SO(2)$ as a subgroup of $\textrm{Aut}(\mathbb{R}^{0,2,1}, d)$ as below.
\[ H := 
\left \{ \left ( 
  \begin{array}{ccc}
   \cos{\theta} & -\sin{\theta} & 0 \\
   \sin{\theta} & \cos{\theta} & 0 \\
    0 & 0 & 1
  \end{array}
\right ) \in \textrm{Aut}(\mathbb{R}^{0,2,1}, d) \
\middle | \ \theta \in \mathbb{R}
\right \}.
\]
Then, the group $H$ is isomorphic to $SO(2)$ as a Lie group. We simply denote $H$ as $SO(2)$. A surface is said to $SO(2)$-invariant if it is invariant under the action of this group. Such surfaces are locally parametrized by
\[ f(u, v) = (x(u)\cos{v}, x(u)\sin{v}, y(u)) \in \mathbb{R}^{0,2,1}, \]
where $x, y$ are real variable functions satisfying $x > 0, \ (x')^{2} + (y')^{2} = 1$. Then, we have 
\[ f_{u} = (x'\cos{v}, x'\sin{v}, y'), \quad f_{v} = (-x\sin{v}, x\cos{v}, 0). \]
Thus, we compute 
\[ g_{11} = (x')^{2}, \quad g_{12} = 0, \quad g_{22} = x^{2}. \]
The non-degeneracy implies $x' \neq 0$. Moreover, since we compute
\begin{eqnarray*}
f_{uu} &=& (x''\cos{v}, x''\sin{v}, y'') = \frac{x''}{x'} f_{u} + \left (-\frac{x''}{x'}y' + y'' \right ) \xi, \\
f_{uv} &=& (-x'\sin{v}, x'\cos{v}, 0) = \frac{x'}{x} f_{v}, \\
f_{vv} &=& (-x\cos{v}, -x\sin{v}, 0) = - \frac{x}{x'} f_{u} + \frac{x}{x'}y' \xi,
\end{eqnarray*}
the coefficients of second fundamental form $h$ hold
\[ h_{11} = -\frac{x''}{x'}y' + y'', \quad h_{12} = 0, \quad h_{22} = \frac{x}{x'}y'. \]
Therefore, we compute that the mean curvature of $SO(2)$-invariant $d$-minimal surfaces is 
\begin{equation}
2\mathcal{H} = g^{ij}h_{ij} = \frac{1}{(x')^{3}}(-x''y' + x'y'') + \frac{y'}{xx'} \equiv 0. \label{inv-min}
\end{equation}
Since $x' \neq 0$, by the coordinate transformation, we can represent $y$ as a function with respect to $x$. Then, the equation (\ref{inv-min}) is equal to the following equation 
\[ \frac{d^{2}y}{dx^{2}} = - \frac{1}{x} \frac{dy}{dx}. \]
By solving the ordinary differential equation, we have
\[ y(x) = C_{1} \log{x} + C_{2} \quad (C_{1}, C_{2} \in \mathbb{R} : \textrm{constants}). \]
Again, when we replace the parameter $x$ with $x(w) = e^{w}$, we get $y(w) = C_{1}w + C_{2}$. In particular, if $C_{1} = 0$, then it is a plane. So, if it is not a plane, by an affine isometry, we obtain 
\[ f(u, v) = (e^{u}\cos{v}, e^{u}\sin{v}, u). \]
The proof is completed.
\end{proof}

\begin{figure}[H]
\begin{center}
\begin{tabular}{c}
 \begin{minipage}{0.5\hsize}
  \begin{center}
   \includegraphics[width=50mm]{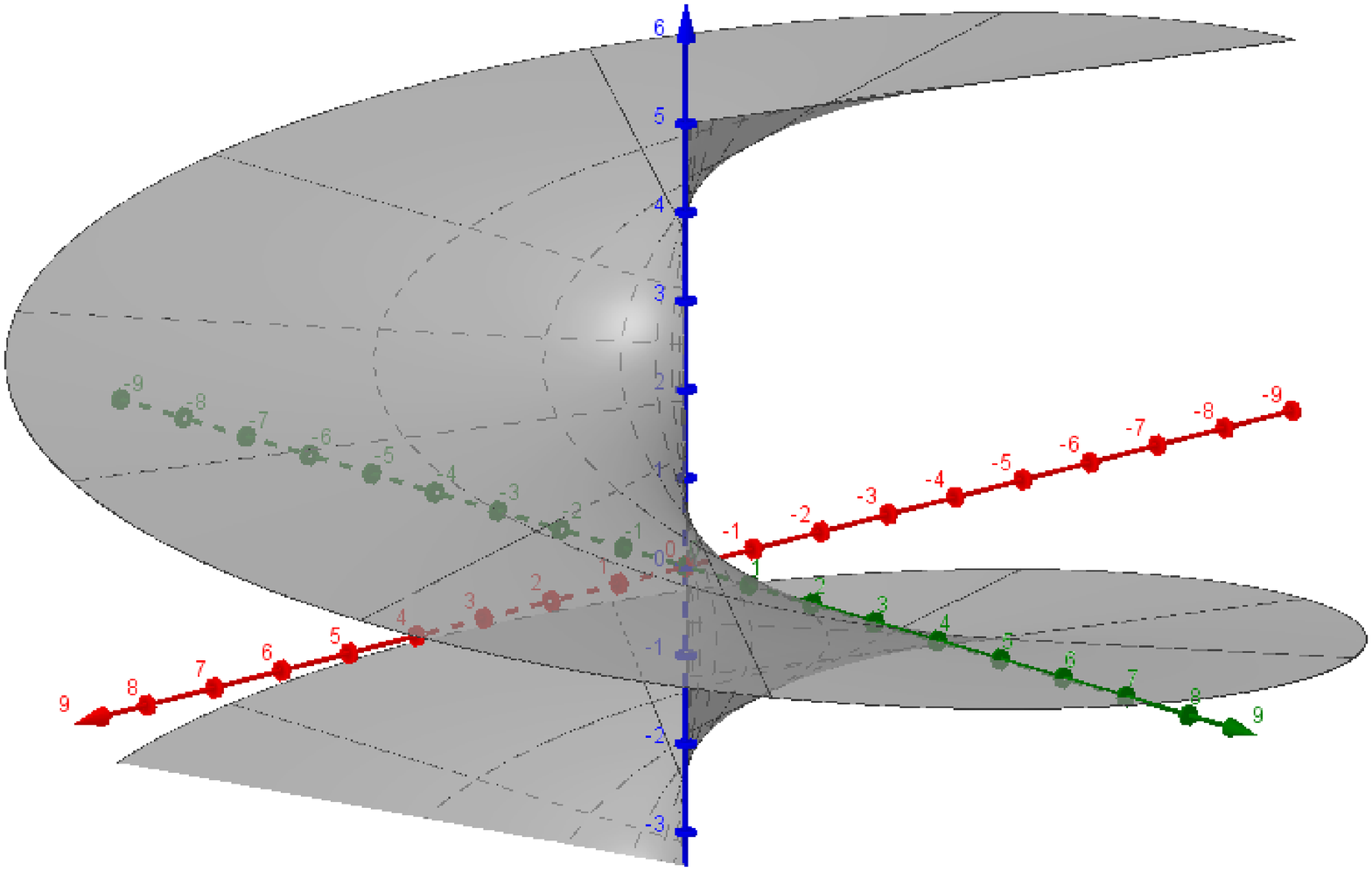}
  \end{center}
  \caption{Elliptic helicoid of the second kind.}
  \label{fig:9}
 \end{minipage}
 \begin{minipage}{0.5\hsize}
  \begin{center}
   \includegraphics[width=55mm]{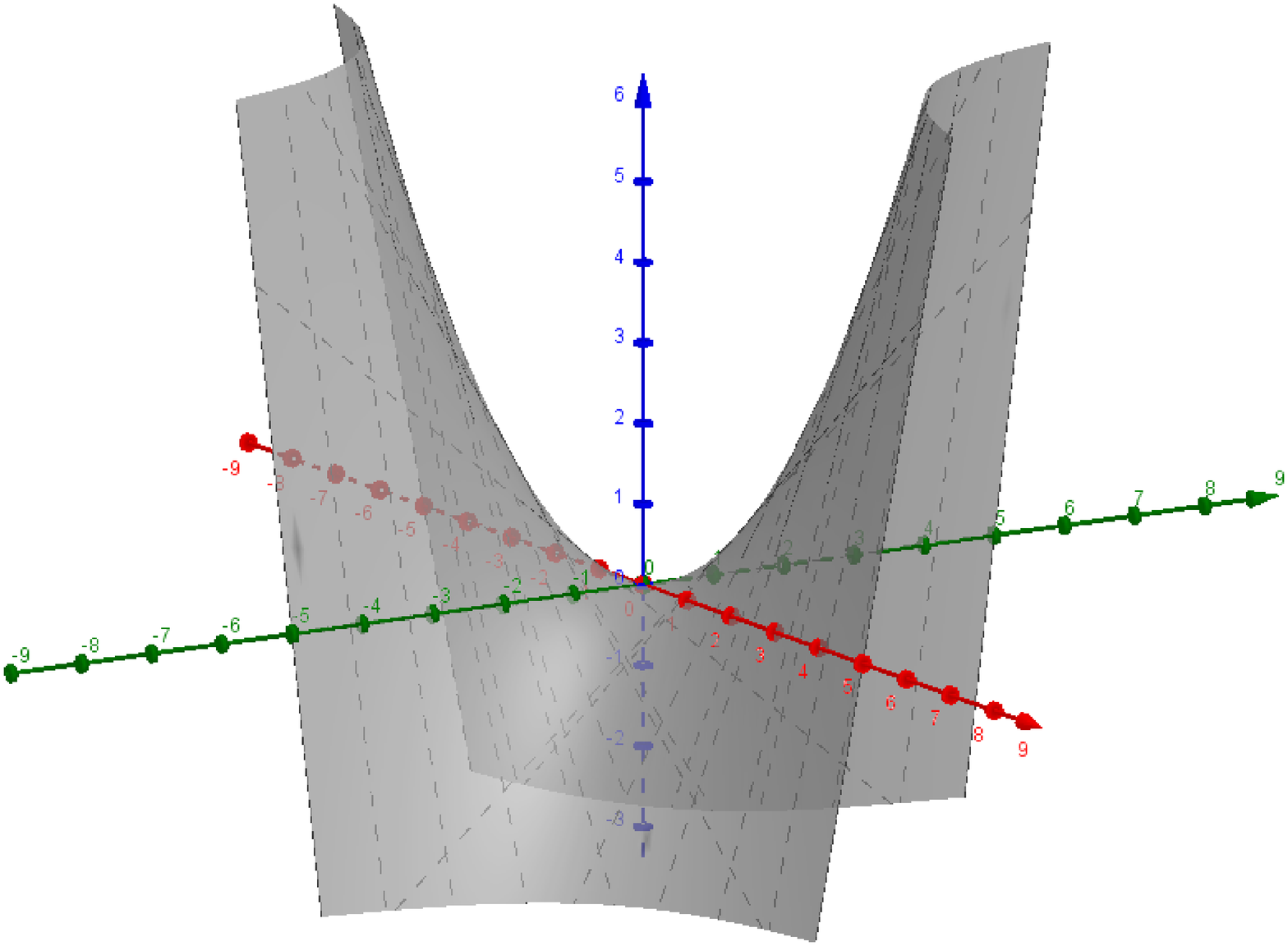}
  \end{center}
  \caption{Minimal hyperbolic paraboloid.}
  \label{fig:10}
 \end{minipage}
\end{tabular}
\end{center}
\end{figure}

\begin{figure}[H]
\begin{center}
\begin{tabular}{c}
 \begin{minipage}{0.5\hsize}
  \begin{center}
   \includegraphics[width=60mm]{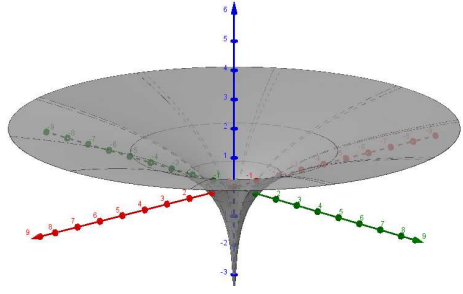}
  \end{center}
  \caption{$d$-minimal rotational surface.}
  \label{fig:11}
 \end{minipage}
\end{tabular}
\end{center}
\end{figure}

\begin{rem} \rm
We recall that non-degenerate surfaces are locally expressed by graph surfaces. However, (a) of Proposition \ref{kindofsurf} is an example which can not be entirely expressed as a graph surface.
\end{rem}

We consider the canonical connection $d$ as a linear connection for $\mathbb{R}^{0, 2, 1}$. 
This connection $d$ is a torsion-free connection which is parallel with respect to the degenerate metric $(\cdot , \cdot)$, i.e. $d$ plays the role of Levi-Civita connection. However, since the metric is degenerate, connections having such properties are not unique. For example, let $\lambda \in \mathbb{R}$ be a real parameter, and we define a tensor field $L_{\lambda} \in \Gamma(S^2T^{\ast}\mathbb{R}^{3})$ as
\[ L_{\lambda}(X, Y) := \lambda \sum_{i, j} X_{i}Y_{j}, \]
where the set $\Gamma(S^2T^{\ast}\mathbb{R}^{3})$ expresses the whole of $(0, 2)$-type symmetric tensor fields over $\mathbb{R}^{3}$ and $X, Y$ are vector fields over $\mathbb{R}^{3}$, and we regard $X$ and $Y$ respectively as vector-valued functions 
\[ X=(X_{1}, X_{2}, X_{3}), Y=(Y_{1}, Y_{2}, Y_{3}). \]
Then, when we put $d^{\lambda} := d + L_{\lambda} \xi$, $d^{\lambda}$ is a flat connection over $\mathbb{R}^{0, 2, 1}$ which has the same properties of Levi-Civita connections.

When we consider $d^{\lambda}$-totally geodesic surfaces defined as the case of $d$, 
non-trivial examples appear, i.e. there exist examples which are not planes (refer to fig.~\ref{fig:2}).

As an example satisfying $h^{\lambda} \equiv 0$ except for planes, we find, for instance, 
\[ F(u, v) = \frac{1}{\lambda} \log{|\lambda u + 1|} - u - v, \]
where $u < - \frac{1}{\lambda}, u > - \frac{1}{\lambda}$.

As a remark, let $\nabla$ be a torsion-free, metric connection on $\mathbb{R}^{0,2,1}$. if all non-degenerate plane are $\nabla$-totally geodesic, then it holds $\nabla = d$, i.e. the case of simply isotropic geometry.

\begin{figure}[H]
\begin{center}
\begin{tabular}{c}
 \begin{minipage}{0.5\hsize}
  \begin{center}
   \includegraphics[width=60mm]{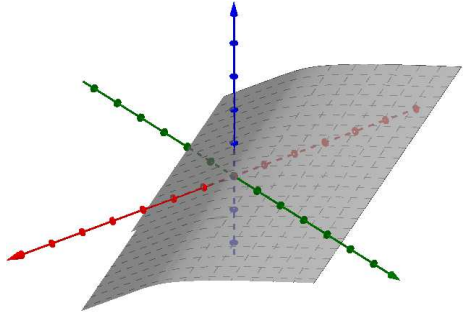}
  \end{center}
  \caption{$d^{\lambda}$-totally geodesic non-plane.}
  \label{fig:2}
 \end{minipage}
\end{tabular}
\end{center}
\end{figure}

\subsection{Representation formula of Weierstrass type for $d$-minimal surfaces}

Let $f : M \rightarrow \mathbb{R}^{0,2,1}$ be a non-degenerate immersion. When we set $f = (f_{1}, f_{2}, f_{3})$, we define Laplacian $\Delta_{g}f$ of $f$ with respect to the induced metric $g$ as Laplacians of each coordinate functions $f_{i} \ (i=1,2,3)$, i.e. 
\[ \Delta_{g}f := (\Delta_{g}f_{1}, \Delta_{g}f_{2}, \Delta_{g}f_{3}). \]

\begin{prop} \label{prop4.2.1} \rm
Let $\mathcal{H}$ be the mean curvature of a non-degenerate immersion $f$. Then, $2\mathcal{H} \xi \in \Gamma(T^{\bot}M)$ is equal to Laplacian $\Delta_{g}f$ of $f$ with respect to the induced metric $g$. In particular, the non-degenerate surface is a $d$-minimal if and only if coordinate functions of $f$ are all harmonic with respect to $g$.
\end{prop}

\begin{proof}
Since $f$ is non-degenerate, there exists a coordinate neighborhood $U$ of $M$ such that the local expression of $f$ is 
\[ f(u, v) = (u, v, F(u,v)) \in \mathbb{R}^{0,2,1}, \]
where $F$ is a function on $U$.
By using this coordinate, we get 
\[ 2\mathcal{H} \xi = (0,0,F_{uu}+F_{vv}) = \Delta_{g}f. \]
The proof is completed.
\end{proof}
In case of graph surfaces, Proposition \ref{prop4.2.1} is equivalent to the formula (8) in \cite{PGM}. 

Next, we prepare some simple lemmas. 

\begin{lem} \label{Wei.1} \rm
For a real two variable function $f(u,v)$, we define a complex function $F(w)$ with respect to the complex variable $w = u + iv$ as
\[ F(w) := \frac{\partial f}{\partial u}(u,v) - i \frac{\partial f}{\partial v}(u,v). \]
Then, $F$ is a holomorphic function if and only if $f(u, v)$ is a harmonic function. 
\end{lem}

\begin{proof}
The Cauchy-Riemann's equations imply.
\end{proof}

\begin{lem} \label{Wei.2} \rm
In $\mathbb{R}^{0,2,1}$, we consider a surface given by 
\[ f(u, v) = ( x(u, v), y(u, v), z(u, v) ) \in \mathbb{R}^{0,2,1}. \]
We define a complex function $\varphi, \psi$ with respect to the complex variable $w = u + iv$ as
\[ \varphi(w) := \frac{\partial x}{\partial u}(u,v) - i \frac{\partial x}{\partial v}(u,v), \quad \psi(w) := \frac{\partial y}{\partial u}(u,v) - i \frac{\partial y}{\partial v}(u,v). \]
Then, the coordinates $(u, v)$ is isothermal if and only if it holds 
\[ \varphi^{2} + \psi^{2} \equiv 0. \]
\end{lem}

\begin{proof}
By direct calculations, we have
\begin{eqnarray*}
\varphi^{2} + \psi^{2} &=& |f_{u}|^{2} - |f_{v}|^{2} - 2 i (f_{u}, f_{v}).
\end{eqnarray*}
This completes the proof.
\end{proof}

\begin{thm} \label{thm4.2.1} \rm
Let $U$ be an open subset of $uv$-plane. In $\mathbb{R}^{0,2,1}$, let $f$ be an immersion on $U$ which is parametrized by $f(u, v)=(x(u, v), y(u, v), z(u, v))$. We assume that $(u, v)$ is the isothermal coordinates and $f$ is $d$-minimal. Then, complex functions $\varphi_{1}, \varphi_{2}, \varphi_{3}$ with respect to the complex variable $w=u+iv$ defined by 
\begin{equation}
\varphi_{1}(w)=\frac{\partial x}{\partial u} - i \frac{\partial x}{\partial v}, \quad \varphi_{2}(w)=\frac{\partial y}{\partial u} - i \frac{\partial y}{\partial v}, \quad \varphi_{3}(w)=\frac{\partial z}{\partial u} - i \frac{\partial z}{\partial v} \label{data1}
\end{equation}
are all holomorphic, and it holds
\begin{equation}
|\varphi_{1}|^2 + |\varphi_{2}|^2 > 0, \quad \varphi_{1}^2 + \varphi_{2}^2 = 0.  \label{data2}
\end{equation}
Moreover, it holds
\[ ( f_{u}, f_{u} ) = ( f_{v}, f_{v} ) = \frac{1}{2}(|\varphi_{1}|^2 + |\varphi_{2}|^2). \]
Conversely, let $U$ be a simply-connected domain on $\mathbb{C}$, and we assume that holomorphic functions $\varphi_{1}(w), \varphi_{2}(w), \varphi_{3}(w)$ satisfy the formula (\ref{data2}). Then, when we set $w = u + iv \in U$, there exists a $d$-minimal surface satisfying the formula (\ref{data1}) such that, for the parametrized expression $f(u, v)=(x(u, v), y(u, v), z(u, v))$, the coordinates $(u, v)$ are isothermal. 
\end{thm}

\begin{proof}
Since $f$ is $d$-minimal, each coordinate functions are harmonic from Proposition \ref{prop4.2.1}.
Thus, by using Lemma \ref{Wei.1}, each $\varphi_{i}$ are holomorphic. And, since $(u, v)$ is isothermal coordinates, it holds $\varphi_{1}^{2} + \varphi_{2}^{2} \equiv 0$ from Lemma \ref{Wei.2}. Next, since we compute 
\[ |\varphi_{1}|^{2} + |\varphi_{2}|^{2} = x_{u}^{2} + y_{u}^{2} + x_{v}^{2} + y_{v}^{2} = |f_{u}|^{2} + |f_{v}|^{2} = 2|f_{u}|^{2} = 2|f_{v}|^{2} > 0, \]
the former of the claim holds. For the latter, we assume that holomorphic functions $\varphi_{1}, \varphi_{2}, \varphi_{3}$ on a simply-connected domain $U$ satisfy the formula (\ref{data2}). We fix a point $w_{0} \in U$ and define a real function $x = x(u, v)$ as
\[ x(u, v) := \textrm{Re} \int_{w_{0}}^{w} \varphi_{1}(w) dw \quad (w = u + iv \in U). \]
This is well-defined since $U$ is simply-connected. When we act on this formula by the differential operator 
\[ \frac{\partial}{\partial u} - i \frac{\partial}{\partial v} = 2 \frac{\partial}{\partial w}, \]
we have
\[ \frac{\partial x}{\partial u} - i \frac{\partial x}{\partial v} = 2 \frac{\partial}{\partial w} \textrm{Re} \int_{w_{0}}^{w} \varphi_{1}(w) dw = \varphi_{1}(w). \]
As above, when we define $y=y(u, v)$ and $z=z(u, v)$, we have
\[ \frac{\partial y}{\partial u} - i \frac{\partial y}{\partial v} = \varphi_{2}(w), \quad \frac{\partial z}{\partial u} - i \frac{\partial z}{\partial v} = \varphi_{3}(w). \]
From Lemma \ref{Wei.1} again, we see that $x(u,v), y(u,v), z(u,v)$ are harmonic functions on $U$. Next, we prove that the mapping $f = f(u, v)$ gives a surface, i.e. two-dimensional manifold. For the purpose of that, we prove that the Jacobi matrix
\[ 
   \left(
    \begin{array}{ccc}
      x_{u} & y_{u} & z_{u} \\
      x_{v} & y_{v} & z_{v} 
    \end{array}
  \right)
\]
is rank two for any point $w \in U$. We prove by using contradiction, i.e. we assume that there is a point $w' \in U$ such that the rank of its Jacobi matrix is less than one. Since we have 
\[ 0 < |\varphi_{1}|^{2} + |\varphi_{2}|^{2} = (x_{u})^{2} + (x_{v})^{2} + (y_{u})^{2} + (y_{v})^{2}, \]
at the point $w'$, we see that either of column vectors 
\[ 
  \left(
    \begin{array}{c}
       x_{u} \\
       x_{v}
     \end{array}
  \right),
\
  \left(
    \begin{array}{c}
       y_{u} \\
       y_{v}
     \end{array}
  \right) \]
is not the zero vector. So, we suppose that the former is not the zero vector. From the assumption of contradiction, since we may set 
\[ \exists \lambda \in \mathbb{R} \ \textrm{s.t.} \ 
  \left(
    \begin{array}{c}
       y_{u} \\
       y_{v}
     \end{array}
  \right)
= 
  \lambda \left (
    \begin{array}{c}
       x_{u} \\
       x_{v}
     \end{array}
  \right),
          \]
by using $\varphi_{2} = \lambda \varphi_{1}$, we compute 
\[ \{ \varphi_{1}(w') \}^{2} + \{ \varphi_{2}(w') \}^{2} = (1 + \lambda^{2})\{ \varphi_{1}(w') \}^{2} \neq 0 \]
at $w'$. This contradicts the formula (\ref{data2}). Thus, since $f$ is a $C^{\infty}$-immersion, $f(u, v) = (x(u,v), y(u,v), z(u,v))$ gives a surface in $\mathbb{R}^{0,2,1}$, and $(u, v) \in U$ is the isothermal coordinates from the condition (\ref{data2}). In particular, $f$ is a $d$-minimal surface satisfying the formula (\ref{data1}).
\end{proof}

\begin{thm}[Weierstrass-type representation formula for $d$-minimal surfaces] \label{thm4.2.2} \rm
Let $U \subset \mathbb{C}$ be a simply-connected domain, and let $F, G$ be holomorphic functions on $U$, where $F$ does not have zero points on $U$. Then, a mapping 
\[ f(u, v) = \textrm{Re} \int_{w} (F, -i F, G) dw \quad (w := u + iv \in U) \]
gives a $d$-minimal surface in $\mathbb{R}^{0,2,1}$, and the coordinates $(u, v) \in U$ are isothermal. Moreover, it holds 
\[ ( f_{u}, f_{u} ) = ( f_{v}, f_{v} ) = |F|^2. \]
Conversely, a $d$-minimal surface in $\mathbb{R}^{0,2,1}$ locally have the expression as above.
\end{thm}

\begin{proof}
For the former of the claim, when we set $\varphi_{1} := F, \varphi_{2} := -iF, \varphi_{3} := G$, it immediately holds from Theorem \ref{thm4.2.1}.
For the latter of the claim, given a $d$-minimal surface,  it is locally considered on a simply-connected domain. From Theorem \ref{thm4.2.1} again, we have the parametrized expression 
\[ f(u, v) = \textrm{Re} \int (\varphi_{1}, \varphi_{2}, \varphi_{3}) dw. \]
Since it satisfies 
\[ |\varphi_{1}|^2 + |\varphi_{2}|^2 > 0, \ \varphi_{1}^2 + \varphi_{2}^2 = 0, \]
setting $F := \varphi_{1}, G := \varphi_{3}$, we obtain the expression which we want.
\end{proof}

Zero points of $G$ correspond with singularities of $d$-minimal surfaces. For example, we see cross-caps on $d$-minimal surfaces. We remark that there exist some singularities not only cross-caps. We state in detail in the next section.

At the end of this section, for Weierstrass type expression formula for $d$-minimal surfaces 
\[ f(u, v) = \textrm{Re} \int_{w} (F, -iF, G) dw \quad (w := u+iv \in U), \]
the function $F$ expresses the induced metric, i.e. it holds $( f_{u}, f_{u} ) = ( f_{v}, f_{v} ) = |F|^2$. On the other hand, the function $G$ is concerned with the second fundamental form $h$ by the following proposition. 

\begin{prop} \rm
Under the situation stated above, it holds
\begin{align*}
h &= \left \{ (\textrm{Re}G)_{u} - \frac{|F|_{u}}{|F|}(\textrm{Re}G) - \frac{|F|_{v}}{|F|}(\textrm{Im}G) \right \}(du^{2} - dv^{2}) \\
&\qquad + \left \{ (\textrm{Re}G)_{v} - \frac{|F|_{v}}{|F|}(\textrm{Re}G) + \frac{|F|_{u}}{|F|}(\textrm{Im}G) \right \}(2dudv) . 
\end{align*}
In particular,
\[ \det{h} = -\left(|G|^{2}_{u} + |G|^{2}_{v}\right) - \left |\frac{G}{F} \right |^{2}\left(|F|^{2}_{u} + |F|^{2}_{v}\right) + 2\left |\frac{G}{F} \right |(|F|_{u}|G|_{u} + |F|_{v}|G|_{v}) . \]
\end{prop}

\begin{proof}
By direct calculations. 
\end{proof}

\begin{rem} \rm 
The pair $(F,G)$ is called a \textit{Weierstrass data}. And, for any $\theta \in \mathbb{R}/{2\pi\mathbb{Z}}$, 
\[ f_{\theta}(s,t) = \cos{\theta}\left(\textrm{Re} \int (F, -iF, G) dw\right) + \sin{\theta}\left(\textrm{Im} \int (F, -iF, G) dw\right) \]
is a $d$-minimal surface in $\mathbb{R}^{0,2,1}$ and this gives an isometric deformation.

In fact, for holomorphic functions $F, G$, it follows 
\[ \textrm{Re} \int_{w_{0}}^{w}(-iF, -F, -iG)dw = \textrm{Im} \int_{w_{0}}^{w} (F, -iF, G)dw. \]
Thus, $d$-minimal surfaces defined by the Weierstrass data $(-iF,-iG)$ correspond with the imaginary part of the formulas defined by the Weierstrass data $(F, G)$. For $\theta \in \mathbb{R}/{2\pi\mathbb{Z}}$, when we consider the $d$-minimal surface whose Weierstrass data is $(e^{-i\theta}F, e^{-i\theta}G)$, the given immersion is called an \textit{associated family} and, when we denote $f_{\theta}$, we have the $S^{1}$-family of mappings. Moreover, we see 
\begin{eqnarray*}
f_{\theta}(u, v) &=& \textrm{Re} \int_{w_{0}}^{w} (e^{-i\theta}F, -ie^{-i\theta}F, e^{-i\theta}G) dw \\
&=& \cos{\theta} \left ( \textrm{Re} \int_{w_{0}}^{w} (F, -iF, G) dw \right ) + \sin{\theta} \left (\textrm{Im} \int_{w_{0}}^{w} (F, -iF, G) dw \right ).
\end{eqnarray*}
In particular, when $\theta = 0, \frac{\pi}{2}$, they correspond with the $d$-minimal surfaces given by the real part and imaginary part from $(F, G)$ respectively. Moreover, for any $\theta \in \mathbb{R}/{2\pi\mathbb{Z}}$, since the induced metric of $f_{\theta}$ satisfies
\[ ((f_{\theta})_{u}, (f_{\theta})_{u}) = ((f_{\theta})_{v}, (f_{\theta})_{v}) = |e^{-i\theta}F|^{2} = |F|^{2}, \ ((f_{\theta})_{u}, (f_{\theta})_{v}) = 0, \]
it gives an isometric deformation between $f = f_{0}$ and $f_{\theta}$. 
We call $f_{\frac{\pi}{2}}$ a \textit{conjugate surface} of $f_{0}$.
\end{rem}

\begin{exa*} \rm
\
\begin{itemize}
\item[(0)] When $(F,G)=(\alpha,\beta)$ ($\alpha,\beta \in \mathbb{C}, \alpha \neq 0$), a non-degenerate plane appears.
\item[(1)] When $(F,G)=(z,1)$, we have 
\[ f_{0}(u, v)= \left (\frac{1}{2}(u^2 - v^2), uv, u \right ), \quad  f_{\frac{\pi}{2}}(u, v) = \left (uv, -\frac{1}{2}(u^2 - v^2), v \right ). \]
These are surfaces which have the self-intersection and both give singularities called as cross-caps at $(u,v)=(0,0)$ (refer to Figure \ref{fig:3}).
\item[(2)] When $(F,G)=(e^{z},1)$, we have 
\[ f_{0}(u, v)=(e^{u}\cos{v}, e^{u}\sin{v}, u), \quad f_{\frac{\pi}{2}}(u, v)=(e^{u}\sin{v}, -e^{u}\cos{v}, v) . \]
$f_{0}$ is the $d$-minimal rotational surface given by Proposition \ref{kindofsurf} (3), and $f_{\frac{\pi}{2}}$ is the elliptic helicoid of the second kind (refer to Figure \ref{fig:9}, Figure \ref{fig:11}).
\item[(3)] When $(F,G)=(1,z)$, we have 
\[ f_{0}(u, v)=\left (u, v, \frac{1}{2}(u^2 - v^2) \right ), \quad f_{\frac{\pi}{2}}(u, v) = (u, -v, uv). \]
These both are minimal hyperbolic paraboloids (refer to Figure \ref{fig:10}).
\end{itemize}
\end{exa*}

\begin{rem} \rm
The above Weierstrass-type representation formula is essentially known in \cite{AP}. However, the formula stated in \cite{AP} does not give singularities on surfaces.  In this sense, Theorem \ref{thm4.2.2} is more complete. On the other hand, We can see isotropic minimal surfaces which have isolated singularities in \cite{PGM}. 
\end{rem}

\subsection{Applications}

\begin{thm} \label{thm4.3.1} \rm
Let $(M, g)$ be a connected, two-dimensional complete Riemannian manifold, and let $f : (M, g) \rightarrow \mathbb{R}^{0,2,1}$ be an isometric immersion. Then, $(M, g)$ is isometric to the canonical two-dimensional Euclidean space $\mathbb{R}^{2}$, and the image of $f$ corresponds with an entire graph 
\[ \{ (u, v, F(u, v) ) \in \mathbb{R}^{0,2,1} \ | \ (u, v) \in \mathbb{R}^{2} \}, \]
where $F$ is a $C^{\infty}$-function on $\mathbb{R}^{2}$.
\end{thm}

\begin{proof}
We define $C^{\infty}$-functions $\alpha, \beta, \gamma$ on $M$ as
\[ f(x) = (\alpha(x), \beta(x), \gamma(x)) \quad (x \in M). \]
We assume that $\mathbb{R}^{2}$ is the canonical Euclidean space which treats $(u, v)$ as the coordinates, and define a $C^{\infty}$-map $f_{0} : (M, g) \rightarrow \mathbb{R}^{2}$ as 
\[ f_{0}(x) := (\alpha(x), \beta(x)) \quad (x \in M). \]
$f_{0}$ is an isometric immersion. 
We prove that $f_{0}$ is the isometric diffeomorphism. We remark that $\textrm{dim}M = \textrm{dim}\mathbb{R}^{2} = 2$ and, from the inverse function theorem, $f_{0}$ is the locally diffeomorphism. Thus, in order to prove that $f_{0}$ is the isometric diffeomorphism, it is sufficient to prove that $f_{0}$ is bijective.

For the surjectivity, since $f_{0}$ is the locally homeomorphism, $f_{0}$ is the open mapping. Thus, $\textrm{Im}f_{0}$ is an open subset of $\mathbb{R}^{2}$. Next, since isometric mappings preserve the geodesic completeness, from Hopf-Rinow's theorem, $(\textrm{Im}f_{0}, du^{2} + dv^{2}) \subset \mathbb{R}^{2}$ is complete, where we consider $\textrm{Im}f_{0}$ as the metric subspace of $\mathbb{R}^{2}$ naturally. Thus, $\textrm{Im}f_{0}$ is a closed subset of $\mathbb{R}^{2}$. Therefore, since $\textrm{Im}f_{0}$ is the open and closed subset of $\mathbb{R}^{2}$, it holds $\textrm{Im}f_{0} = \mathbb{R}^{2}$, i.e. $f_{0} : M \rightarrow \mathbb{R}^{2}$ is surjective.

For the injectivity, we denote the Riemannian distance with respect to the metric $g$ by $d_{M}$. For arbitrary points $x, y \in M$ which are distinct, since $(M, g)$ is complete, there exist a short geodesic $\delta : [0, 1] \rightarrow M$ such that $\delta(0) = x, \delta(1) = y$. Moreover, since $f_{0}$ is isometric, $f_{0} \circ \delta : [0, 1] \rightarrow \mathbb{R}^{2}$ is a geodesic in $\mathbb{R}^{2}$ which connects $f_{0}(x)$ and $f_{0}(y)$. For a curve $c$, when we denote the length of $c$ by $L(c)$, we see 
\[ 0 < d_{M}(x,y) = L(\delta) = L(f_{0} \circ \delta) = |f_{0}(x) - f_{0}(y)|_{\mathbb{R}^{2}}. \]
This implies $f_{0}(x) \neq f_{0}(y)$, i.e. $f_{0} : M \rightarrow \mathbb{R}^{2}$ is injective. As a remark, we use the fact that geodesics in $\mathbb{R}^{2}$ are straight lines for the last equation above.

In summary, since we obtain that $f_{0} : M \rightarrow \mathbb{R}^{2}$ is a locally isometric diffeomorphism and bijection, it is an isometric diffeomorphism, that is, $(M, g)$ is isometric to the canonical two-dimensional Euclidean space $\mathbb{R}^{2}$. We denote the inverse of $f_{0}$ by $\phi : \mathbb{R}^{2} \rightarrow M$. For any $(u, v) \in \mathbb{R}^{2}$, we have 
\begin{eqnarray*}
f(\phi(u, v)) &=& (\alpha(\phi(u,v)), \beta(\phi(u,v)), \gamma(\phi(u,v))) \\
&=& ((f_{0} \circ \phi)(u,v), (\gamma \circ \phi)(u,v)) = (u, v, F(u,v)),
\end{eqnarray*}
where $F := \gamma \circ \phi$ is a $C^{\infty}$-function on $\mathbb{R}^{2}$. Therefore, the image of $f$ is the entire graph expressed by a function $F$ on $\mathbb{R}^{2}$.
\end{proof}

\begin{cor} \label{cor4.1} \rm
Let $f : M^2 \rightarrow \mathbb{R}^{0,2,1}$ be a connected, complete $d$-minimal surface. Then, the image of $f$ corresponds with the entire graph 
\[ \{ (u, v, \psi(u, v) ) \in \mathbb{R}^{0,2,1} \ | \ (u, v) \in \mathbb{R}^{2} \}, \]
where $\psi$ is a harmonic function on $\mathbb{R}^{2}$.
\end{cor}

\begin{proof}
From Proposition \ref{prop4.2.1} (2), it follows immediately.
\end{proof}

\begin{cor} \label{cor4.2} \rm
Let $M$ be a connected, compact two-dimensional manifold, i.e. a connected closed surface. Then, there exist no non-degenerate immersion $f : M \rightarrow \mathbb{R}^{0,2,1}$.
\end{cor}

\begin{proof}
We prove the corollary by contradiction. We assume that there exist a non-degenerate immersion $f : M \rightarrow \mathbb{R}^{0,2,1}$. When we denote the induced metric by $g$, $(M, g)$ is a connected, compact Riemannian manifold. In particular, it is complete. From Theorem \ref{thm4.3.1}, as we have a homeomorphism $M \cong \mathbb{R}^{2}$, this contradicts the compactness of $M$. 
\end{proof}

Let $f : M \rightarrow \mathbb{R}^{0,2,1}$ be a non-degenerate immersion, and let $h$ be its second fundamental form. Then, we recall that the Gauss-Codazzi equation of the non-degenerate immersion is given by the formula (\ref{gauss-codazzi}), i.e.
\[ (\nabla_{X}h)(Y, Z) = (\nabla_{Y}h)(X, Z) \quad (X, Y, Z \in \Gamma(TM)). \]
By using the flat local coordinates $(u, v)$, the formula (\ref{gauss-codazzi}) is equivalent to
\begin{equation}
(h_{11})_{v} = (h_{12})_{u}, \quad (h_{22})_{u} = (h_{12})_{v}, \label{codazzi}
\end{equation}
where $h_{ij}$ are coefficients of $h$. 

\begin{thm}[The fundamental theorem of non-degenerate surfaces, \cite{Sach}, Theorem 8.8] \label{thm4.3.2} \rm
Let $U \subset \mathbb{R}^{2}$ be a simply-connected domain, and let $(u,v)$ be coordinates on $U$. And, let $h_{11}, h_{12}, h_{22}$ be $C^{\infty}$-functions on $U$. Then, there exist, up to affine isometry, a non-degenerate immersion whose the induced metric and the second fundamental form are 
\[ du^{2} + dv^{2}, \quad h_{11}du^{2} + 2h_{12}dudv + h_{22}dv^{2} \]
respectively if and only if the functions $h_{ij}$ satisfy the Gauss-Codazzi equation (\ref{codazzi}) of the non-degenerate surface.
\end{thm}

From now on, we consider four-dimensional Minkowski space $\mathbb{R}^{4}_{1}$ which equips with the Lorentzian metric 
\[ \langle \cdot, \cdot \rangle_{1} := -dx_{1}^2 + dx_{2}^2 + dx_{3}^2 + dx_{4}^2, \]
where $(x_{1}, x_{2}, x_{3}, x_{4})$ is the canonical coordinates of $\mathbb{R}^{4}$. 
We deal with spacelike surfaces only, i.e. we require that the induced metric of surfaces is positive definite. 

A surface $M$ is called \textit{zero mean curvature} if it holds $\vec{H} \equiv 0$, where $\vec{H}$ is the mean curvature vector field of $M$, and a surface $M$ is called \textit{flat} if it holds $K \equiv 0$, where $K$ is the Gaussian curvature of $M$. We abbreviate zero mean curvature to ZMC. 

On the other hand, $\mathbb{R}^{0,2,1}$ is isometrically embedded in $\mathbb{R}^4_{1}$ by the natural way. In fact, the following mapping 
\begin{equation}
\iota : \mathbb{R}^{0,2,1} \ni (x, y, z) \mapsto (z, x, y ,z) \in \mathbb{R}^{4}_{1} \label{isom_emb} 
\end{equation}
is an isometric embedding. 

\begin{rem} \rm
We give one of the motivations of studying flat and zero mean curvature surfaces. We firstly remark that flat minimal submanifolds in $n$-dimensional Euclidean space $\mathbb{R}^{n}$ and  spacelike flat ZMC surfaces in three dimensonal Minkowski space $\mathbb{R}^{3}_{1}$ are totally geodesic. On the other hand, there exists timelike flat ZMC surfaces in $\mathbb{R}^{3}_{1}$ (\cite{Sato}). Thus, we are interested in the case of spacelike surfaces. 

Next, spacelike flat ZMC surfaces in four-dimensional Minkowski space $\mathbb{R}^{4}_{1}$ are not always planes. In particular, we also remark that spacelike flat ZMC surfaces in four-dimensional semi-Euclidean space $\mathbb{R}^{4}_{2}$ equipped with the neutral metric are totally geodesic again. 
\end{rem}

\begin{thm} \label{thm4.3.3} \rm
Let $f : M^2 \rightarrow \mathbb{R}^{4}_{1}$ be an immersion which gives a non-totally geodesic, connected spacelike flat ZMC surface, and let $h$ be the second fundamental form of $M$. We define a subset $E$ of $M$ as 
\[ E:=\{ x \in M \ | \ h_{x} = 0 \}. \]
Then, it holds the following assertions:
\begin{itemize}
\item[(1)] $M \setminus E$ is an open dense subset of $M$, and it is connected.
\item[(2)] The normal bundle of $M$ is flat, i.e.  the normal curvature $R^{\bot} \equiv 0$.
\item[(3)] $M$ is, by an isometry of $\mathbb{R}^{4}_{1}$, immersed in $\mathbb{R}^{0,2,1} \subset \mathbb{R}^{4}_{1}$, and it is a $d$-minimal surface. 
\end{itemize}
\end{thm}

\begin{proof}
The claim (1) is easily proved that $E$ is a closed subset of $M$. 
Since $M$ is not totally geodesic, we obtain $\overset{\circ}{E} = \emptyset$. Moreover, we see that the set $E$ do not have accumulation points, that is, the set $E$ is a discrete subset of $M$ which is made of isolated points. Thus, $M \setminus E$ is an open dense subset of $M$. And, since $M$ is connected and $E$ is discrete, it is proved for $M \setminus E$ to be connected.

For (2), see Corollary 1.2 in \cite{AP}. The claim (3) is proved by using Proposition \ref{prop4.2.1} and Proposition 3.5 in \cite{AP}
\end{proof}

\begin{rem} \rm
The set $E$ is a discrete subset of $M$ consisted of isolated points. As an example which satisfies $E \neq \emptyset$, when we define a $C^{\infty}$-immersion $f : \mathbb{R}^2 \rightarrow \mathbb{R}^{0,2,1} \subset \mathbb{R}^{4}_{1}$ as 
\[ f(u, v) := (u^3 - 3uv^2, u, v, u^3 - 3uv^2), \]
it is a spacelike flat ZMC surface which satisfies $h=0$ at the origin $(0, 0)$ only. 
\end{rem}

Let $f : M \rightarrow \mathbb{R}^{0,2,1}$ be a $d$-minimal surface. Then, by the isometric embedding $\iota$ given (\ref{isom_emb}), we see that $M$ is a spacelike flat ZMC surface in $\mathbb{R}^4_{1}$. 
$M$ is a spacelike flat surface since $\iota$ is an isometric embedding. To show that $M$ is ZMC, we directly calculate the mean curvature vector field of $M$. By using a harmonic function $\varphi$, since we can locally expressed by 
\[ f(u, v) = (u, v, \varphi(u, v)), \]
from the composition of $\iota$, we have 
\[ (\iota \circ f)(u, v) = (\varphi(u, v) , u, v, \varphi(u, v)). \]
Thus, we compute that the mean curvature vector field $\vec{H}$ is 
\[ 2\vec{H} = (\iota \circ f)_{uu} + (\iota \circ f)_{vv} = (\varphi_{uu} + \varphi_{vv})(1, 0, 0, 1) \equiv 0. \]
Therefore, we obtain the following corollary. 

\begin{cor} \label{cor} \rm
Let $X$ be the set of the isometric classes of spacelike flat ZMC surfaces in $\mathbb{R}^{4}_{1}$, and let $Y$ be the set of equivalence classes of $d$-minimal surfaces in $\mathbb{R}^{0,2,1}$ by a subgroup 
\[ K:=  \left. \left \{
    \left(
    \begin{array}{ccc}
      \multicolumn{2}{c}{\raisebox{-2.0ex}{\LARGE{\textit{T}}}} & 0 \\
       & &0 \\
      0 & 0& c
    \end{array}
  \right)
 \ \right | \ c \neq 0, \ T \in O(2) \right \} \ltimes \mathbb{R}^3 \subset \textrm{Aut}(\mathbb{R}^{0, 2, 1}, d). \] 
Then, except for planes, we have that $X$ and $Y$ are in one-to-one correspondence. 
\end{cor}

\begin{proof}
It is obvious as long as we remark that this subgroup $K$ corresponds to the subgroup of isometric group of $\mathbb{R}^{4}_{1}$ which preserves the degenerate subspace $\mathbb{R}^{0,2,1} \subset \mathbb{R}^{4}_{1}$.
\end{proof}

Regarding minimal surfaces in $\mathbb{R}^{3}$, maximal surfaces in $\mathbb{R}^{3}_{1}$ and $d$-minimal surfaces in $\mathbb{R}^{0,2,1}$, we have 
\begin{eqnarray*}
&& \{ \textrm{minimal surfaces, maximal surfaces and} \ d\textrm{-minimal surfaces} \} \\
&\subset& \{ \textrm{spacelike ZMC surfaces in} \ \mathbb{R}^{4}_{1} \}.  
\end{eqnarray*}
In fact, for the spaces $\mathbb{R}^{3}$ and $\mathbb{R}^{3}_{1}$, there exist isometric embeddings defined by
\begin{eqnarray*}
&\mathbb{R}^{3}& \ni (x, y, z) \mapsto (0, x, y, z) \in \mathbb{R}^{4}_{1}, \\
&\mathbb{R}^{3}_{1}& \ni (x, y, z) \mapsto (x, y, z, 0) \in \mathbb{R}^{4}_{1}
\end{eqnarray*}
respectively. Therefore, we see that there quite fruitfully exist ZMC surfaces in $\mathbb{R}^{4}_{1}$.

In general, singularity points appear in $d$-minimal surfaces. Refer to the figures from \ref{fig:3} to \ref{fig:8} as such examples. From the Whitney's criterion, a cross-cap appears in Figure \ref{fig:3}, and from the Saji's criterion (\cite{Sa}), a $D_{4}^{-}$-type singularity appears in Figure \ref{fig:7}. Other singularities have been not identified and classified. 

At the end of this paper, we give a table which compares properties among each surfaces. We assume the connectedness of surfaces; 
\begin{table}[htb]
  \begin{center}
     \begin{tabular}{|c|c|c|c|} \hline
          & min. & max. & $d$-min. \\ \hline
        Compact & $\nexists$ & $\nexists$ & $\nexists$ (Cor. \ref{cor4.2}) \\ \hline
        Entire graph & Planes only & Planes only & $\exists$ (Prop. \ref{kindofsurf}) \\ \hline
        Singularity & $\nexists$ & $\exists$ (\cite{FSUY}) & $\exists$ \\ \hline
        Complete & $\exists$ & Planes only & $\exists$ (Thm. \ref{thm4.3.1}) \\ \hline
        Gaussian curvature & $\leq 0$ & $\geq 0$ & $\equiv 0$ \\ \hline
     \end{tabular}
  \end{center}
\caption{} \label{tab:1}
\end{table}

where, in terms of singularity, the symbol $\exists$ expresses that singularities appear, and in terms of otherwise,  $\exists$ expresses that there exist such surfaces which are not planes. In addition to, the abbreviations min., max. and $d$-min. are minimal surfaces in $\mathbb{R}^{3}$, maximal surfaces in $\mathbb{R}^{3}_{1}$ and $d$-minimal surfaces in $\mathbb{R}^{0,2,1}$ respectively.

\section*{Acknowledgment}
The author would like to express his deepest gratitude to his advisor, Professor
Takashi Sakai. The author is also very grateful to Shintaro Akamine and Luis Carlos Barbosa da Silva for their comments and valuable advices.

\newpage
$\bullet$ Examples of isolated singularities whose the rank of Jacobi matrix is one for $C^{\infty}$-mappings $f$ giving $d$-minimal surfaces. 
\begin{figure}[H]
\begin{center}
\begin{tabular}{c}
 \begin{minipage}{0.5\hsize}
  \begin{center}
   \includegraphics[width=40mm]{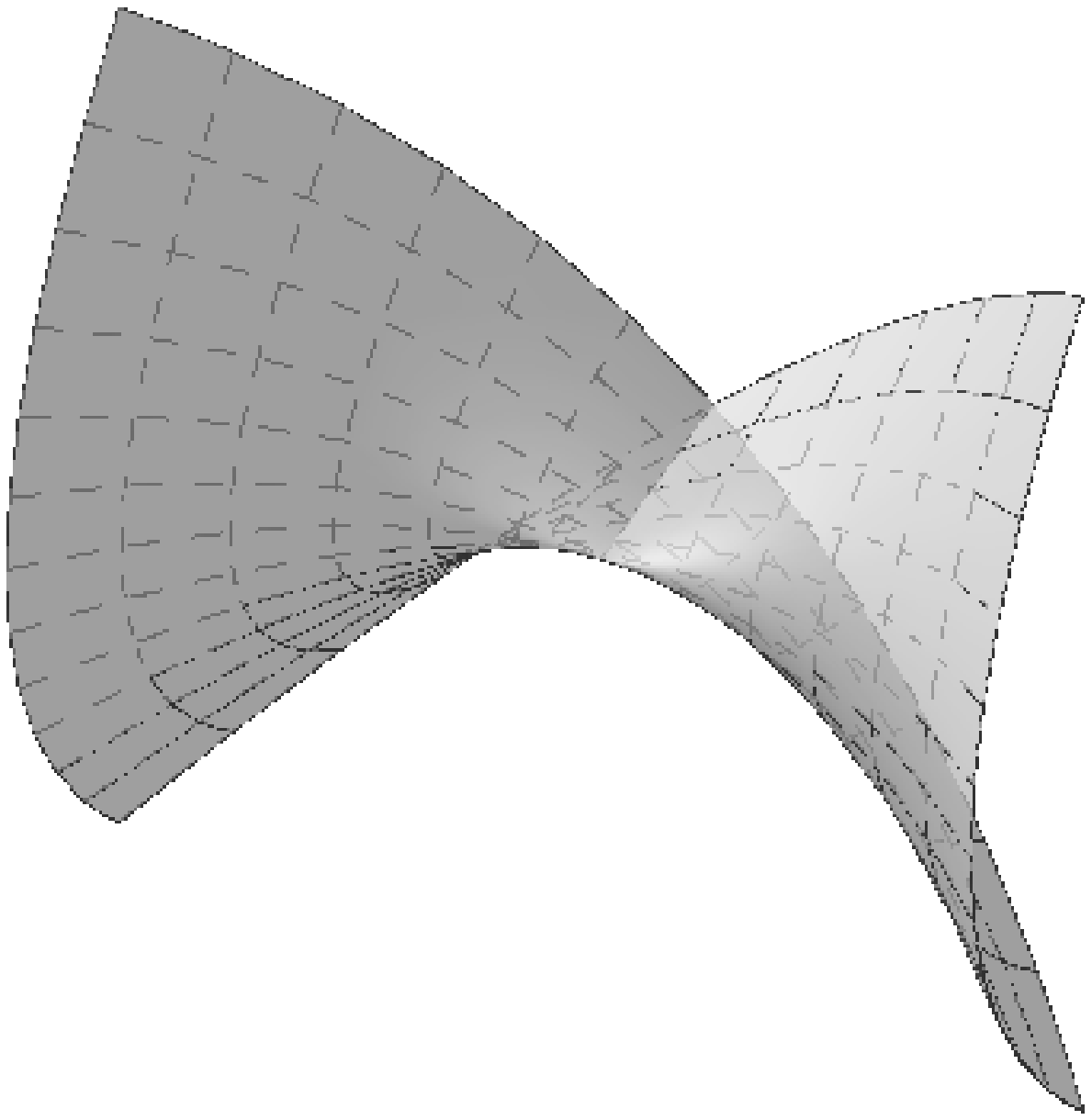}
  \end{center}
  \caption{$(F,G)=(z,1)$.}
  \label{fig:3}
 \end{minipage}
 \begin{minipage}{0.5\hsize}
  \begin{center}
   \includegraphics[width=45mm]{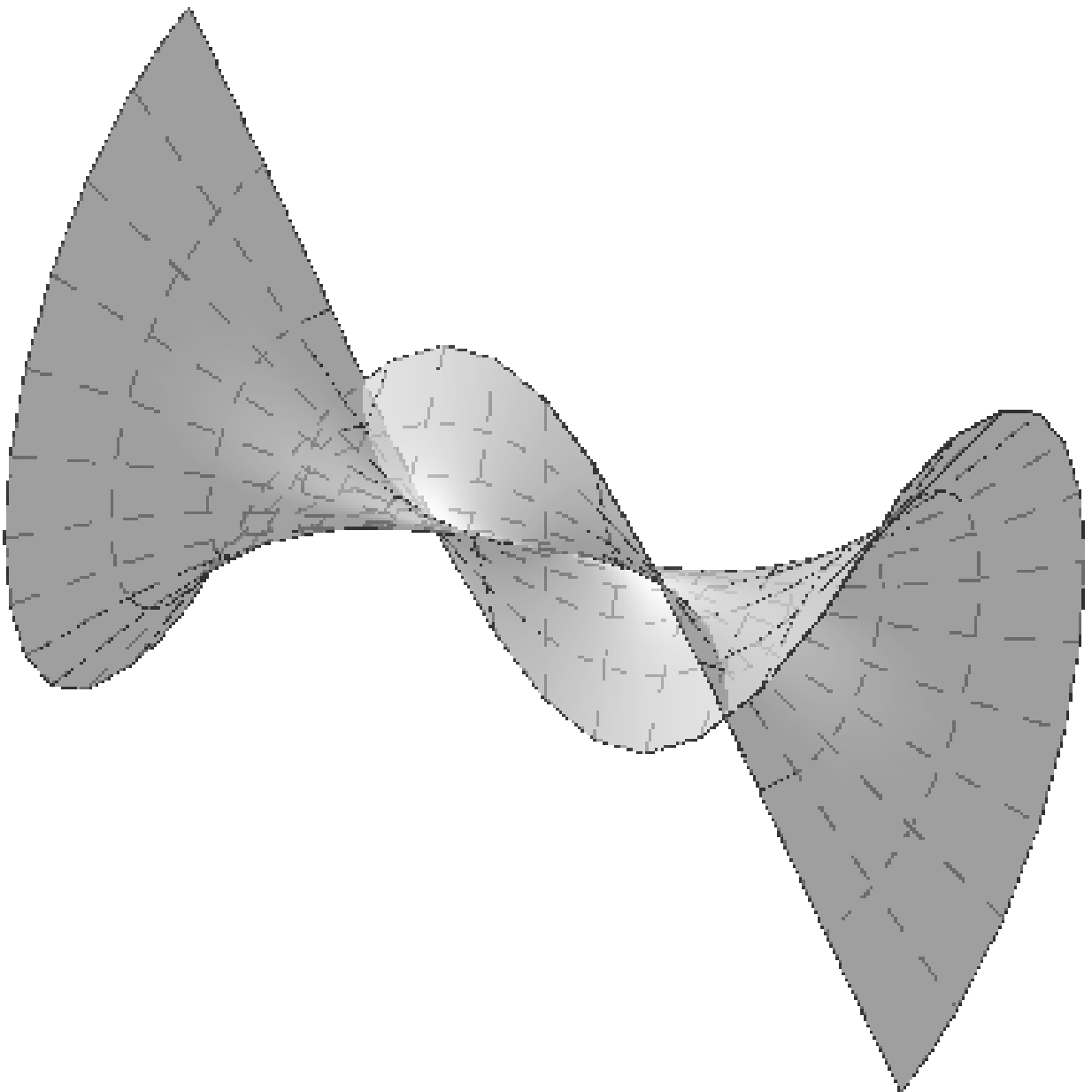}
  \end{center}
  \caption{$(F,G)=(z^2,1)$.}
  \label{fig:4}
 \end{minipage}
\end{tabular}
\end{center}
\end{figure}
\begin{figure}[H]
\begin{center}
\begin{tabular}{c}
 \begin{minipage}{0.5\hsize}
  \begin{center}
   \includegraphics[width=55mm]{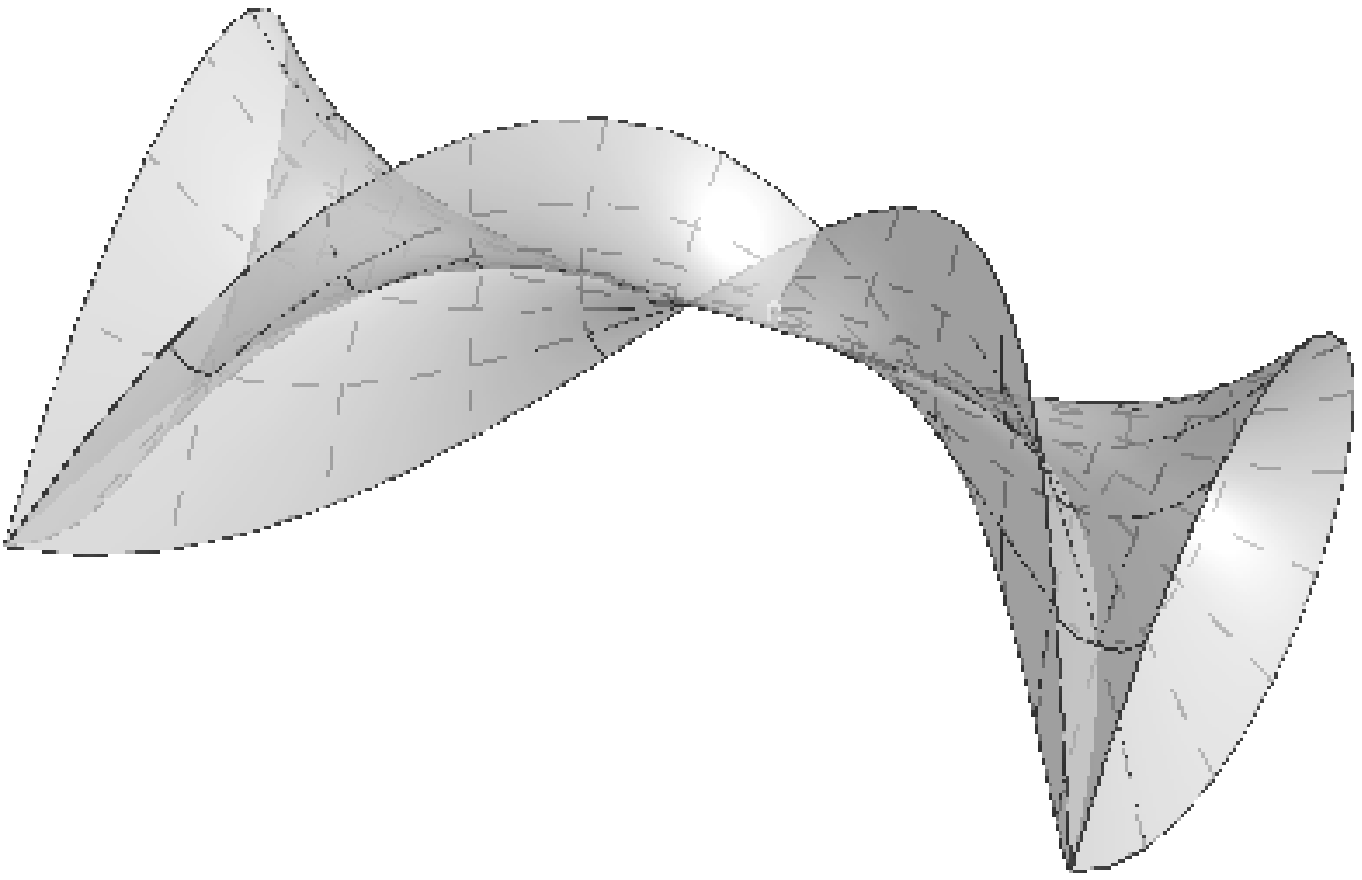}
  \end{center}
  \caption{$(F,G)=(z^3,1)$.}
  \label{fig:5}
 \end{minipage}
 \begin{minipage}{0.5\hsize}
  \begin{center}
   \includegraphics[width=55mm]{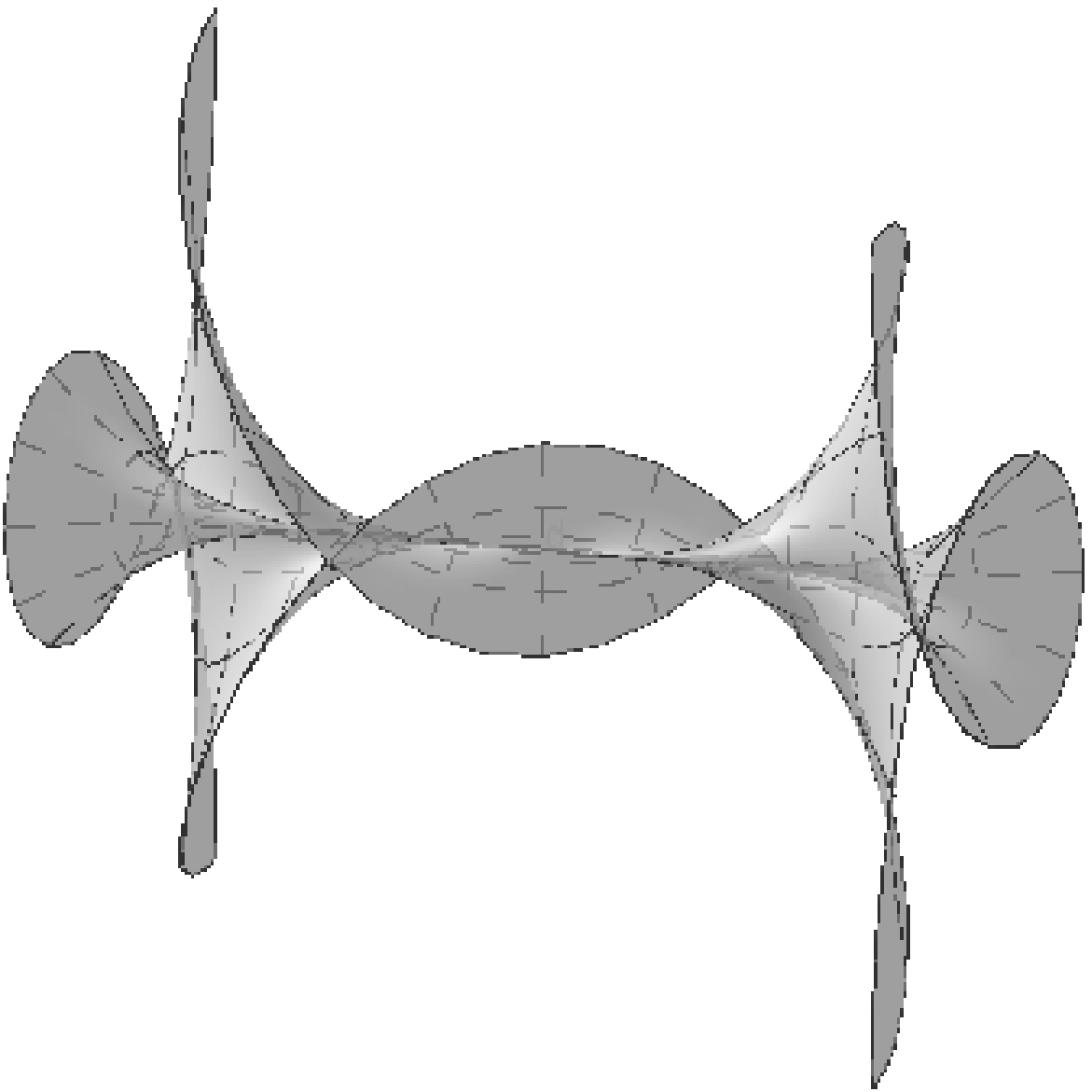}
  \end{center}
  \caption{$(F,G)=(z^4,1)$.}
  \label{fig:6}
 \end{minipage}
\end{tabular}
\end{center}
\end{figure}

$\bullet$ Examples of isolated singularities whose the rank of Jacobi matrix is zero for $C^{\infty}$-mappings $f$ giving $d$-minimal surfaces. 
\begin{figure}[H]
\begin{center}
\begin{tabular}{c}
 \begin{minipage}{0.5\hsize}
  \begin{center}
   \includegraphics[width=30mm]{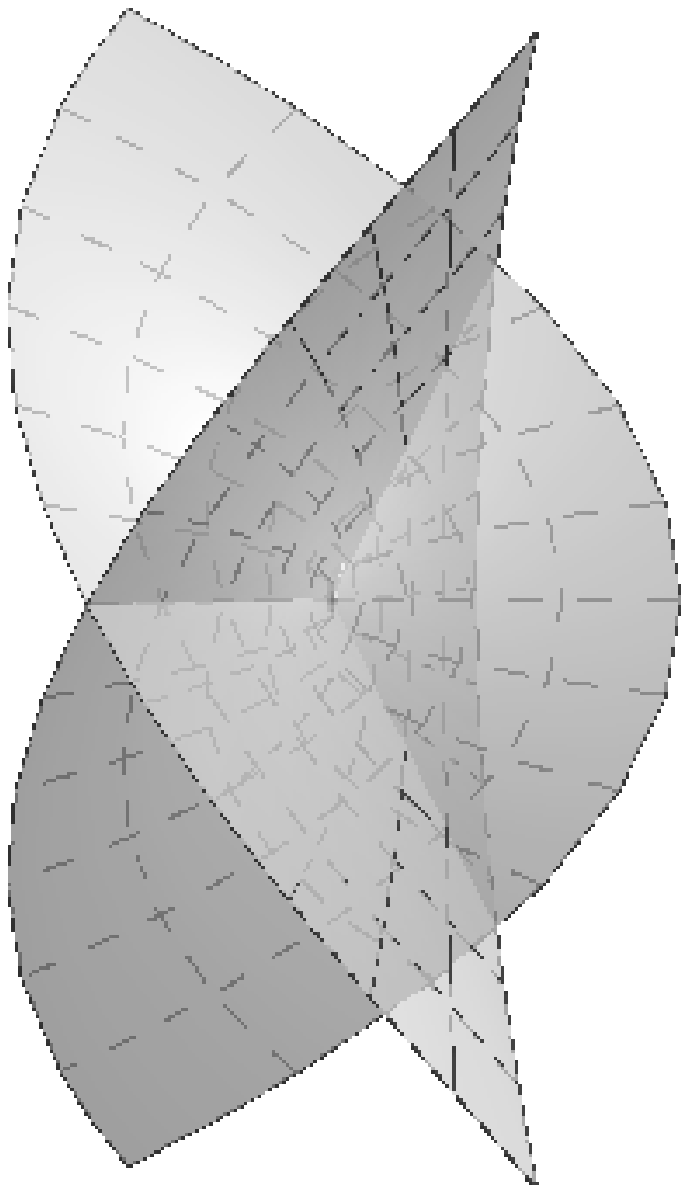}
  \end{center}
  \caption{$(F,G)=(z,z^2)$.}
  \label{fig:7}
 \end{minipage}
 \begin{minipage}{0.5\hsize}
  \begin{center}
   \includegraphics[width=40mm]{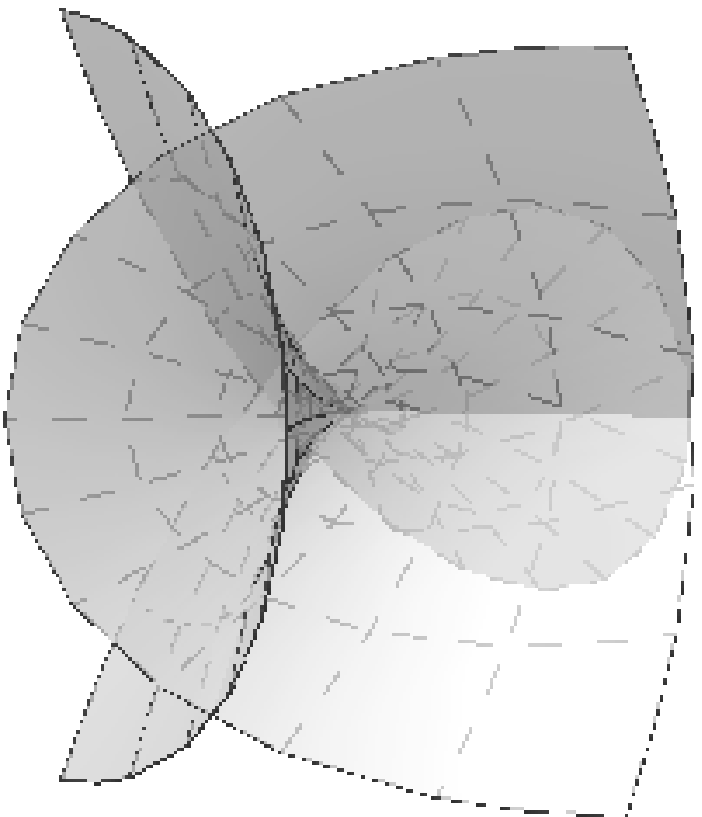}
  \end{center}
  \caption{$(F,G)=(z^2,z)$.}
  \label{fig:8}
 \end{minipage}
\end{tabular}
\end{center}
\end{figure}

\end{document}